\documentclass[11pt]{article}

\usepackage[margin=1in]{geometry}
\usepackage{graphicx}

\usepackage{amssymb}
\usepackage{amsmath}
\usepackage{amsthm}

\newtheorem{definition}{Definition}

\newtheorem{proposition}{Proposition}
\newtheorem{lemma}{Lemma}

\newtheorem{claim}{Claim}

\usepackage{lineno}

\usepackage[figuresright]{rotating}
\usepackage{pdflscape}
\usepackage{booktabs}
\usepackage{multirow}
\usepackage{subcaption}
\usepackage{footmisc}

\usepackage{pgfplots}
\pgfplotsset{compat=1.18}
\usepgfplotslibrary{colormaps}
\usepackage{tikz}
\usetikzlibrary{fit}
\usetikzlibrary{decorations.pathreplacing}
\usetikzlibrary{positioning}

\tikzset{
    startstop/.style={rectangle, rounded corners, minimum width=4cm, minimum height=1.5cm, text centered, draw=black, fill=blue!30},
    process/.style={rectangle, minimum width=3cm, minimum height=1.3cm, text centered, draw=black, fill=orange!30},
    processMOD/.style={rectangle, minimum width=3cm, minimum height=1.3cm, text centered, draw=black, fill=red!30},
    finalprocess/.style={rectangle, minimum width=3cm, minimum height=1.3cm, text centered, draw=black, fill=blue!30},
    groupbox/.style={dashed, draw=black, inner xsep=10pt, inner ysep=10pt},
    arrow/.style={thick,->,>=stealth}
}

\usepackage{siunitx}
\newcommand{\FA}[1]{\textcolor{black}{#1}}

\usepackage{authblk}

\usepackage{natbib}

\usepackage{hyperref}
\usepackage{algorithm}
\usepackage{algpseudocode}

\title{The Capacitated $p$-Location Problem with Territorial Coverage Constraint: Efficient Solution and Case Study}

\author[1]{Felipe Albuquerque\thanks{felipe.albuquerque@univ-avignon.fr}}
\author[1]{Rosa Figueiredo\thanks{rosa.figueiredo@univ-avignon.fr}}
\author[2]{Cyrille Genre-Grandpierre\thanks{cyrille.genre-grandpierre@univ-avignon.fr}}
\affil[1]{LIA - Laboratoire Informatique d'Avignon, Avignon Université}
\affil[2]{UMR 7300 ESPACE, Avignon Université}
\date{}

\begin{document}

\maketitle

\begin{abstract}

This paper studies the Capacitated $p$-Location Problem (C$p$LP) and its extensions incorporating equity considerations. In location science, the $p$-Median problem ($p$MP) is a classical model that selects $p$ facilities from a finite set of candidates to serve a set of customers while minimizing total allocation costs. The C$p$LP, which includes the Capacitated $p$-Median Problem (C$p$MP) and its relaxed variant (C$p$MP$^r$), extends the $p$MP by incorporating capacity constraints on facilities.
We formalize the C$p$LP with Territorial Coverage Constraints (C$p$LP-TC), an extension that enforces equity across spatial units, and generalize it to a multi-scale variant (C$p$LP-MTC) that enforces equity simultaneously across nested spatial scales. 
We present a strengthened Integer Linear Programming (ILP) formulation with valid inequalities that enables the exact solution of medium-sized instances. To tackle larger problems, we adapt the Random Sampling Spatial Voting (RSSV) heuristic, originally proposed for the $p$MP, into a competitive open-source matheuristic that combines a heuristic reduction phase with the strengthened ILP formulation. The resulting method remains flexible, requires minimal parameter tuning, and is accessible to non-specialist users.
Computational experiments on a new open-source benchmark instance set, built from a French regional case study, demonstrate the 
effectiveness of both the exact and heuristic approaches.
Beyond their computational performance,
the results quantify the trade-offs between efficiency and territorial equity under different equity constraints, providing practical insights for equitable location planning.

\end{abstract}

\noindent\textbf{Keywords:} Capacitated $p$-Location, Coverage Constraints, Integer Programming, Matheuristic


\section{Introduction}
\label{sec:Intro}
Location Science is an interdisciplinary field (geography, economics and computer sciences) focusing on selecting optimal sites for human activities to address social, economical and environmental challenges~\citep{Laporte2019}. In location science, mathematical optimization is widely used to derive efficient spatial allocations, with service location problems receiving increasing attention since the 2000s~\citep{celik2020comparative}. 
The classic problem is the $p$-Median Problem ($p$MP), which selects exactly $p$ facilities from a finite set of candidates to serve a set of customers, each assigned to a single facility, while minimizing total allocation costs (distance for example)~\citep{ReVelle1970}. The Capacitated $p$-Median Problem (C$p$MP) and its relaxed assignment variant (C$p$MP$^r$), are two well-known extensions of the $p$-MP. The first introduces facility capacity constraints (each facility can only handle a given amount of demand) and is known to be NP-hard~\citep{GAREY1974}, with early studies by \citet{mulvey1984solving}.
The second (C$p$MP$^r$) allows demand to be split among multiple facilities~\citep{farahani2009facility}.
In this paper, we refer to both variants as the Capacitated $p$-Location Problem (C$p$LP), a unified problem for our study.

The C$p$LP follows an efficiency logic, that leads to allocating resources to areas with the highest demand, in order to minimize overall allocation costs. However, in real-world applications, particularly those related to spatial planning for which the political dimension is important, considerations related to equity emerge in addition to efficiency. Equity is a fundamental principle of spatial planning, linked to the concept of spatial justice ~\citep{kunzmann1998planning}. The rationale is to ensure that everyone, wherever they are, including in sparsely populated or marginal areas, enjoys living conditions that are as similar as possible, particularly in terms of accessibility~\citep{van2022accessibility}. 
Thus, in concrete land-use planning issues, particularly in the fields of education, health, and culture \citep{HAN2023, wolf2021efficiency, northridge2011urban}, it is politically impossible to settle for simple logic, with, on the one hand, efficiency favoring areas of high demand in opposition to the demand for spatial justice and, on the other, equity ensuring fair territorial coverage, but without taking sufficient account of economic considerations related to efficiency. In practice, empirical compromises are found, based on the balance of power between the various players, but without any assurance regarding their optimality. 
The question of how to find optimal locations mixing efficiency and fairness, which is the responsibility of planning authorities, therefore remains largely unanswered, ~\citep{BATTA2014, enayati2020optimal}.

One way to address this issue is to modify traditional efficiency allocation models. 
Certain approaches modify the objective function. For the $p$MP, the $p$-Center problem \citep{suzuki1996p} is a variation that minimizes the maximum assignment distance, while multi-objective formulations are also used to balance efficiency and equity through Pareto optimality \citep{kong2025extended}.
However, these approaches are not designed to ensure a guaranteed minimum level of territorial coverage across predefined spatial units.
Alternatively, 
one can retain an efficiency-oriented objective function while adding explicit equity constraints. This is the approach adopted in this study.
We define new problem formulations by extending the C$p$LP to promote equity through territorial coverage constraints, ensuring, whenever possible, that at least one facility is located within each spatial unit. Considering a predefined territorial partitioning, we introduce the C$p$LP with Territorial Coverage Constraints (C$p$LP-TC). When coverage is required across multiple territorial divisions at different scales, we define the C$p$LP with Multi-Scale Territorial Coverage Constraints (C$p$LP-MTC).

This extension is motivated by the fact that spatial planning is not only technical but also political. Territorial divisions correspond to administrative or governance units, and competition between them implies that equity cannot be assessed solely at an individual level. For example, even if individuals of a population of a given territorial unit have good access to a facility located just outside its unit, from a political point of view this may be unacceptable if their unit of origin  
has no facility, and is therefore considered to be poorly equipped. Thus, location models must account for territorial boundaries to reflect this political dimensions.

To solve location problems, practitioners overwhelmingly rely on commercial tools like ArcGIS~Pro\footnote{\label{note:esri}\url{www.esri.com/en-us/home}} to deal with simple location--allocation scenarios. However, these tools face scalability and flexibility limitations~\citep{CHEN2021}, and they cannot accommodate the territorial constraints mentioned above, as they are restricted to a fixed set of predefined problems. \FA{Alternatively}, for small instances, existing models for the C$p$LP can be adapted and solved exactly through Integer Linear Programming (ILP). In large-scale scenarios however, efficient heuristic methods are required to obtain high-quality solutions within practical time and memory limits.
This paper therefore describes a Random Sampling Spatial Voting (RSSV) matheuristic,
designed to handle all C$p$LP variants considered in this study. The method combines a heuristic component, which reduces instance size, with an exact ILP component. A key advantage of RSSV is its adaptability to different contexts and minimal parameter tuning. This is an essential feature for decision makers, since algorithms with many parameters are often impractical in applied settings~\citep{gwalani2021evaluation}.

In this context, the main contributions of this work are:
\begin{itemize}
    \item Extending the C$p$LP to incorporate spatial equity through territorial coverage constraints, together with a strengthened ILP formulation using valid inequalities.
    \item Developing a competitive open-source RSSV matheuristic for solving the C$p$LP, C$p$LP-TC, and C$p$LP-MTC.
    \item Demonstrating the effectiveness of the RSSV method on a French regional case study through computational experiments.
    \item \FA{Providing a new benchmark instance set with a stepwise procedure for generating additional instances.} 
\end{itemize}



After a literature review in Section~\ref{sec:Lit}, the paper is organized as follows. Section~\ref{sec:ProbDef} defines the Capacitated $p$-Location Problem and its extensions with territorial and multi-scale coverage constraints. Section~\ref{sec:SolMet} presents the ILP formulation, tightening strategy, and RSSV matheuristic. Section~\ref{sec:CompExp} reports computational experiments for a real-world case study from the \textit{Provence-Alpes-Côte d’Azur} (PACA) region in France, focused on cinema services. Additional experiments on literature instances and ArcGIS Pro comparisons are presented in Appendix~\ref{sec:appendix-experiments}. Finally, Section~\ref{sec:Conclusion} summarizes the main findings and future research directions.


\section{Literature review for C$p$LP}
\label{sec:Lit}

In the $p$MP, once the $p$ facilities are selected, each customer is assigned to its nearest facility in polynomial time~\citep{Kariv1979}. As a result, customer demand is never split, and the resulting solutions exhibit a natural clustering structure.
Introducing capacities leads to the C$p$MP, where the assignment step becomes a separate NP-hard problem, known as the Generalized Assignment Problem~\citep{fisher1986multiplier},
increasing problem complexity. Early studies exploited the clustering nature of $p$MP solutions, modeling the C$p$MP as the Capacitated Clustering Problem (CCP)~\citep{mulvey1984solving}. In contrast, the relaxed version, C$p$MP$^r$, allows demand splitting. This, removes the clustering property, and has received comparatively less attention. On the other hand, it exhibits a defining feature of Multiple Allocation Facility Location Problems (MFLPs) \citep{Laporte2019}, in which the number $p$ of selected facilities is not fixed.

Exact solution methods for the C$p$MP include branch-and-bound, column generation, branch-and-price, and cutting planes. These methods build upon the same classic ILP formulation for the problem. 
The author in~\citep{pirkul1987efficient} proposed one of the first branch-and-bound algorithms using a Lagrangian relaxation, solving instances up to 100 nodes (customers and potential facility sites). Subsequent works extended the approach to larger instances: 
\citet{baldacci2002new} 
used set-partitioning formulations for up to 200 nodes, and 
\citet{ceselli2005branch} 
applied branch-and-price with column generation. 
\citet{boccia2008cut} 
employed Fenchel cuts to reduce integrality gaps, solving problems with up to 402 nodes.

Among heuristic methods for the C$p$MP, the classical approach by \citet{mulvey1984solving} alternates between assigning customers to the $p$ selected facilities and updating facility locations. The method uses a greedy regret-based strategy, where the regret of each client is the difference between its distances to the first and second nearest medians, reflecting how much worse the second-best assignment would be. Building on this idea, \citet{koskosidis1992clustering} proposed new initialization methods and regret definitions.
 
\citet{osman1994capacitated} 
introduced randomly generated instances (up to 100 vertices) and a hybrid metaheuristic combining Simulated Annealing and Tabu Search. \citet{Lorena2004} proposed a heuristic combining Lagrangian/surrogate relaxation with column generation, applied to real instances from Brazil (up to 402 nodes), and a large-scale instance with 3038 nodes, based on TSPLIB. \citet{diaz2006hybrid} presented a benchmark for Spain (737 nodes) and hybrid heuristics based on Scatter Search and Path Relinking. Other notable heuristics include GRASP with Adaptive Memory Programming~\citep{AHMADI2005}, VNS~\citep{Fleszar2008}, Clustering Search~\citep{chaves2008clustering}, and hybrid grouping evolutionary algorithms~\citep{landa2012comparative}.

Matheuristics combine the rigor of mathematical programming with heuristic flexibility~\citep{boschetti2024contemporary}. 
\citet{stefanello2015matheuristics} 
proposed IRMA, an Iterated Reduction Matheuristic for the C$p$MP, which iteratively reduces the ILP size through variable elimination heuristics and applies partial optimization to the reduced model. Starting from a randomized primal heuristic solution~\citep{mulvey1984solving}, IRMA solves the smaller ILP and applies post-optimization if needed. Experiments on four literature benchmark sets and new large instances (up to 4461 nodes, following 
\citet{Lorena2004}) 
have demonstrated the efficiency of the IRMA approach. Later, another matheuristic included a selective crossover genetic algorithm with ILP post-processing~\citep{janovsikova2017hybrid}. A large-scale approach by 
\citet{gnagi2021matheuristic} 
handled instances up to 500\,$\mathrm{k}$ nodes by combining global and local optimization phases using ILP on reduced subproblems and $k$-$d$ trees for efficient spatial data management. IRMA remains the best-performing method on literature instances, while 
\citet{gnagi2021matheuristic} 
is currently the most effective method with an accessible code for very large-scale problems. \FA{Most recently, \citet{gjergji2026large} 
proposed a large neighborhood search and a study of hyper-heuristics for the C$p$MP, using an exact MIP solver in the repair phase. Their approaches report lower average GAP values than previous works, but unlike
\citet{gnagi2021matheuristic}, their implementation is not publicly available.}


\FA{Districting Problems~\citep{Kalcsics2019} address a related but distinct question: how to partition a set of basic units (e.g., census tracts, streets, or customer accounts) into contiguous, compact, and balanced territories, often assigning a facility or center to each. This class of problems arises in political redistricting, sales territory design, and service districting, among other applications, and is typically solved via location-allocation methods, set-partitioning approaches, or geometric construction heuristics. While this line of work shares with ours the goal of achieving spatially equitable outcomes, it differs in a fundamental way: districting problems builds the territorial partition itself, as part of the solution, subject to criteria such as balance and compactness. Our work instead takes the territorial divisions as fixed, given in advance as administrative or political boundaries, and pursues equity through territorial coverage constraints.
}

The authors in~\cite{li2011covering} defined a response-time threshold for serving customers, which can be interpreted as a type of coverage constraint in location problems.
In contrast, our study focuses on territorial coverage, ensuring that each spatial unit contains at least one service. To the best of our knowledge, the first work addressing a similar idea of territorial coverage constraints for the $p$MP was presented by \FA{\cite{revelle1989algorithm}, who proposed a two-stage algorithm for facility location in districted regions: a $p$MP is solved independently within each district to evaluate the trade-off between the number of facilities and the resulting weighted distance, and a greedy procedure then allocates the available facilities across districts, ensuring that every district is served by at least one facility. This idea was further formalized by} \cite{church1990regionally}, who introduced additional constraints imposing minimum and maximum limits on the number of facilities to be installed in each region and proposed a Lagrangian relaxation. This work was later extended by \cite{MURRAY1997} to account for capacities. In our study, the C$p$LP with territorial coverage aims to ensure that service allocation maximizes the number of spatial units covered, using additional constraints. Moreover, it is extended to cover multiple types of spatial units simultaneously.

\section{Definitions and mathematical formulation}
\label{sec:ProbDef}

This section defines the Capacitated $p$-Location Problem (C$p$LP), encompassing both the Capacitated $p$-Median Problem (C$p$MP) and its relaxed form (C$p$MP$^r$) (Section~\ref{subsec:C$p$LP}). We then introduce territorial coverage constraints to balance efficiency and equity (Section~\ref{subsec:covConst}), and a multi-scale extension to ensure service coverage across hierarchical territorial divisions (Section~\ref{subsec:Multi_Scale_coverage}).



\subsection{Capacitated $p$-Location Problem}
\label{subsec:C$p$LP}

The $p$-median problem is defined by a set of customers $I$, a set of potential facility locations $J$, and a parameter $p$.
For each pair $(i, j)$ with $i \in I$ and $j \in J$, there is an allocation cost $cost(i, j)$, typically a distance between customer $i$ and location $j$. The objective is to select $p$ locations from $J$ and assign each customer to one of them, so as to minimize the total allocation cost.

In practical applications, each customer $i$ has a demand weight $W_i$, and each location $j$ a capacity limit $R_j$, which corresponds to the Capacitated $p$-Median Problem (C$p$MP). A commonly used ILP formulation for the C$p$MP follows the approach introduced by~\cite{ReVelle1970} and employs two types of decision variables, $y_j$ and $x_{ij}$, defined for each $i \in I$ and $j \in J$. The variable $y_j$ takes the value 1 if a facility is at $j \in J$ and 0 otherwise, whereas $x_{ij}$ takes the value 1 if customer $i \in I$ is assigned to the facility at $j \in J$, and 0 otherwise. The ILP formulation is as follows.
    \begin{align}
        \quad\min & \sum_{i\in I}\sum_{j \in J} cost(i,j) x_{ij} \label{cPMP_obj} \\
        \text{s.t.} \quad &  \sum_{j \in J} x_{ij} = 1, \quad \forall i \in I \label{cPMP_SatifiedCustomer},\\[5pt] 
        & \sum_{j \in J} y_{j} = p,
        \label{cPMP_Exact-p}
        \\[5pt]
        & \sum_{i \in I} W_{i} x_{ij} \leq R_{j} y_{j}, \quad \forall j \in J, 
        \label{cPMP_Capacity}
        \\[5pt]
        & y_{j} \in  \{0,1\},\ x_{ij} \in \{0,1\} ,\quad \forall i \in I, \forall j \in J. 
        \label{cPMP_x_y}
    \end{align}
Constraints (\ref{cPMP_SatifiedCustomer}) ensure that each customer is assigned to exactly one facility, while the total number $p$ of facilities to be installed is specified by constraint (\ref{cPMP_Exact-p}). Constraints (\ref{cPMP_Capacity}) require that each facility's capacity limit is respected, also prohibiting allocations to potential locations  that were not chosen for facilities. Constraints (\ref{cPMP_x_y}) 
define the domains of all variables in the formulation. The objective function (\ref{cPMP_obj}) minimizes the cost of assigning each customer $i$ to a selected facility $j$. 

Depending the service type, we consider the relaxed C$p$MP, denoted C$p$MP$^r$, in which $x_{ij} \in [0,1]$ represents the fraction of demand at $i$ assigned to facility $j$. We denote the Capacitated $p$-Location Problem (C$p$LP) as a unified framework where each instance is defined by $\langle J, I, p, W, R \rangle$, encompassing both C$p$MP and C$p$MP$^r$.

\subsection{Territorial coverage constraints}
\label{subsec:covConst}

Our focus is on applying the C$p$LP to determine service locations within a geographical area $A$.
The set of customers is then denoted as  
$I=\{1, \ldots, m\}$ defining the geographical area $A$ as a set of $m$ distinct spatial units.
The set of locations $J$ represents $|J|$ points within this geographical area.
We also consider a set of $m^s$ geographical subareas $S = \{S_1, S_2, \ldots, S_{m^s}\}$, which is a partition of $I$, i.e. $S_k \cap S_{k'}= \emptyset$ for $k \neq k'$, and $I = \cup_{k=1}^{m^s} S_k$. Let $J(S_k) \subseteq J$ denote the subset of candidate facility locations within the geographical subarea $S_k$.

Next, we define the territorial constraints incorporated in the C$p$LP. These constraints aim to improve territorial coverage by maximizing the number of covered subareas while respecting the facility-opening limit $p$. They require that at least one facility be opened in each subarea, or in as many subareas as possible when $p$ is smaller than the number of subareas $m_s$.
If $p \geq m^s$, the territorial constraint
\begin{equation} \sum_{j \in J(S_k)} y_{j} \geq 1, \quad k \in \{1,2, \ldots, m^s\} \label{cover_all_subar}, \end{equation}
ensures that at least one facility is located in each subarea. On the other hand, if $p \leq m^s$, the constraint
\begin{equation} \sum_{j \in J(S_k)} y_{j} \leq 1, \quad k \in \{1,2, \ldots, m^s\} \label{cover_must_subar}, \end{equation}
ensures that at least $p$ subareas will be covered. 

\textcolor{black}{We refer to the problem incorporating the territorial coverage constraints as C$p$LP-TC. Since C$p$LP is a special case of C$p$LP-TC obtained by setting $m^s = 1$, every instance of C$p$LP can be be reduced in polynomial time to an instance of C$p$LP-TC. Because C$p$LP is NP-hard~\cite{Kariv1979}, C$p$LP-TC is NP-hard. Moreover, the two simple polynomial reductions in Lemmas~\ref{lem:cpmp_tc_hard_p_gt_m} and~\ref{lem:cpmp_tc_hard_p_leq_m} (Appendix~\ref{app:hardness_cpmptc}) show that this NP-hardness holds for both $p > m^s$ and $p \leq m^s$, and therefore for any values of $m^s$ and $p$.
}
\subsection{Multi-Scale Territorial coverage constraints}
\label{subsec:Multi_Scale_coverage}

A territory is typically organized into various hierarchical divisions, each serving a distinct purpose, such as political, electoral, or administrative functions. In service location, coverage constraints may apply at each level of this hierarchy individually, or across combined levels. We define this variant as C$p$LP-MTC, which stands for the 
C$p$LP
with Multi-Scales Territorial Coverage Constraints.

Let $t$ be the number of distinct territorial divisions, with each $l \in \{1, \ldots, t\}$ defining a partition of the geographical area $A$ in $m^{l}$ subareas, denoted as $S^{l} = \{S_{1}^{l}, S_{2}^{l}, \ldots, S_{m^l}^{l}\}$.

\begin{definition}
\label{def:multiscalePartition}
 The collection of these territorial divisions defines a multi-scale partition of $I$, denoted as $P(I)=\{S^l\}_{l\in\{1, \ldots, t\}}$ and satisfying the following properties: 
 \begin{itemize}
\item[(i)] For each $l \in \{1, \ldots, t\}$, $S^l$ is a partition of $I$.
\item[(ii)] For each $l \in \{2, \ldots, t\}$, $k \in \{1, \ldots, m^l\}$,  a unique $q \in \{1, \ldots, m^{l-1}\}$ exists such that $S^l_k \subseteq S^{l-1}_q$.
\end{itemize}
\end{definition}

Property (ii) in the above definition ensures that the partitions in $P(I)$ are hierarchically related: finer partitions (with higher values of $l$) further subdivide the subareas defined by coarser partitions (with lower values of $l$). Note also that $l=1$ has the fewest subareas and $l=t$ the most, with $1 \le m^l \le m$ for all levels $l$.

The multi-scale coverage constraints aim to maximize coverage across the 
$t$ territorial divisions by ensuring as many subareas as possible are covered. Following the approach described in the previous section, where coverage depends on the relationship between $p$ and the number of subareas, this is achieved by including $t$ territorial coverage constraints, given by~(\ref{cover_all_subar}) or~(\ref{cover_must_subar}). However, including just one or two of these constraints is sufficient.

\begin{claim}
Consider $l\in  \{2, \ldots, t\}$ and assume $m^{l-1}\leq p$. The validity of constraints (\ref{cover_all_subar}) for $S^{l-1}$ ensures their validity for all $t\in\{1,\ldots,l-1\}$.
\end{claim}   

\begin{claim}
Consider $l\in  \{1, \ldots, t-1\}$ and assume $p\leq m^l$.  The validity of constraints (\ref{cover_must_subar}) for $S^l$, ensures their validity for $t\in\{l+1,\ldots,t\}$.
\end{claim}

This logic follows directly from Definition~\ref{def:multiscalePartition} of a multi-scale partition and lead to the straightforward proposition stated next. \FA{An illustrative example is shown in Figure~\ref{fig:multiscales_coverage_solutions_example} (Appendix~\ref{sec:appendix-figures}).}
\begin{proposition}
Let $\langle J, I, p, W, R \rangle$ be an instance of the C$p$LP, and let $P(I)$ be a multi-scale partition of $I$ with $m^{\bar{l}-1}\leq p\leq m^{\bar{l}}$, written for $1\leq \bar{l}\leq t$. The validity of territorial coverage constraints (\ref{cover_all_subar}) or (\ref{cover_must_subar}), for each partition $S\in P(I)$, follows from the validity of (\ref{cover_all_subar}) for $S^{\bar{l}-1}$ and (\ref{cover_must_subar}) for $S^{\bar{l}}$.
\end{proposition}
    



\section{Solution methods}
\label{sec:SolMet}

In this section, we present the methods developed to solve the previously introduced problems. As stated in the introduction, our main objective is to design approaches that efficiently handle large-scale instances, require minimal parameter tuning, and easily accommodate additional constraints.
For exact solutions, we rely on the mathematical formulations described earlier, strengthened with valid inequalities and formulation improvements (Section~\ref{subsec:form_strength}).
Complementing these exact methods, we introduce a RSSV matheuristic (Section~\ref{subsec:rssv}) applicable to all problem variants.




\subsection{Tightening the Formulation}
\label{subsec:form_strength}

Here we describe four constraint classes used to strengthen the formulations.
\paragraph{Disaggregated constraints:} As suggested in the literature \citep{Marin2019}, when aggregated constraints~(\ref{cPMP_Capacity}) are present in location problems, an alternative strategy is to combine them with the inclusion of the following disaggregated ones:
\begin{equation}
x_{ij} \leq y_j, \quad \forall i \in I, \forall j \in J.
\label{eq:valid_ineq_disagg}
\end{equation}

These inequalities enforce the constraint that a customer can only be assigned to an open facility. To avoid overloading the model with an excessive number of constraints, it is advisable to restrict its inclusion to the facility \( j^* \) that offers the minimum supply cost for each customer \( i \in I \), i.e., $
j^* = \arg\min_{j \in J} \{ \text{cost}(i,j) : \text{cost}(i,j) > 0 \}.
$
As argued by~\cite{Marin2019}, this selective use of valid inequalities 
strengthens the formulation, yielding tighter lower bounds.

\paragraph{Lifted $p$-cover constraints:}
Introduced in~\citep{Aardal95}, $p$-cover inequalities are classical ones for a C$p$LP. It states that if we can identify a set of \( p \) locations whose total capacity is strictly less than the total demand, then it is impossible for all \( p \) of these locations to be selected in a feasible solution. Formally, let \( J' \subseteq J \) be a set of locations such that \( |J'| = p \) and $\sum_{j \in J'} R_{j} < \sum_{i \in I} W_{i},$ then the following valid inequality holds:
\begin{equation}
\sum_{j \in J'} y_{j} \leq p - 1.
\label{eq:valid_ineq_lift_pcover}
\end{equation}

When assuming territorial coverage constraints~\eqref{cover_must_subar} in the C$p$LP-TC problem defined for a partition $S$ of $I$, these constraints can be lifted as follows.
Remember that when \( p \leq m^S \), each subarea $S_k\in S$ must be covered by at least one selected facility. The $p$-cover constraint~\eqref{eq:valid_ineq_lift_pcover} can be strengthened through a lifting strategy that takes the territorial coverage constraint into account.

Consider a subset \( J' \subseteq J \) defining a constraint~\eqref{eq:valid_ineq_lift_pcover}, i.e. such that:
 \( |J'| = p \) and 
\( \sum_{j \in J'} R_{j} < \sum_{i \in I} W_{i} \). Additionally, assume that each \( j \in J' \) belongs to a distinct subarea denoted $S_{k(j)}$.  For each $j\in J'$, let \( \bar{J}_{k(j)} \subset J(S_{k(j)}) \) denote a subset including more locations from the same subarea as $j$ such that $R_j \geq R_{j'}$, for each $j'\in \bar{J}_{k(j)}$. 
Clearly, the sum of the maximum capacities from each one of these \( p \) subareas is less than or equal to the total demand $\sum_{i \in C} W_i$. Then, even with the expanded set \(\cup_{j\in J'} \bar{J}_{k(j)} \), the demand cannot be satisfied, i.e.,
\begin{equation}
\sum_{j \in J'} \sum_{j' \in \bar{J}_{k(j)}} y_{j'} \leq p - 1.
\label{eq:lifted_coverineq}
\end{equation}

\paragraph{Location upper bound constraints:}
We now consider the C$p$LP-TC problem with \( p \geq m^S \), i.e. when $p$ is sufficient to cover all subareas. Under the coverage constraint~\eqref{cover_all_subar}, we can derive an upper bound on the number of facilities in each subarea, based on the surplus \( p - m^S \), which represents the additional locations available after assigning one to each subarea. For each \( S_k \in S = \{S_1, S_2, \ldots, S_{m^S}\} \), assuming $|J(S_k)| > p - m^S + 1$,
\begin{equation}
\sum_{j \in J(S_k)} y_{j} \leq p - m^S + 1.
\label{eq:valid_ineq_upper_limit_zone}
\end{equation}

This reasoning extends to any subset of subareas \( K' \subset S \). If \( K' \) satisfies
\(
\sum_{S_k \in K'} |J(S_k)| > p - m^S + |K'|,
\)
the total number of facilities across the subareas in \( K' \) is then bounded by:
\begin{equation}
\sum_{S_k \in K'} \sum_{j \in J(S_k)} y_{j} \leq  p - m^S + |K'|.
\label{eq:valid_ineq_upper_limit_zone_set}
\end{equation}

In particular, when the subset \( K' \) contains exactly \( m^S - 1 \) distinct subareas, inequality~\eqref{eq:valid_ineq_upper_limit_zone_set} reduces to inequality~\eqref{eq:lifted_coverineq}.

\paragraph{Distance constraints:} Consider a maximum distance limit \( \mathrm{D} \), which restricts customer assignments to facilities located within a predefined radius. This limitation can be easily added to the previous formulations, by writing the customer satisfaction constraint~\eqref{cPMP_SatifiedCustomer} as,
\begin{equation}
\sum_{j \in J : d_{ij} \leq \mathrm{D}} x_{ij} = 1, \quad \forall i \in I.
\label{cPMP_SatifiedCustomer_Dmax}
\end{equation}

This modification is not a valid inequality, as it restricts rather than strengthens the feasible region. If no facility is available within radius \( \mathrm{D} \) for a customer \( i \), the instance becomes infeasible, and an inappropriate choice of 
 \( \mathrm{D} \) may also exclude optimal or promising solutions. Conversely, a properly selected \( \mathrm{D} \) can greatly reduce the feasible search space, improving computational efficiency by pruning poor solutions. Many heuristics use this strategy in location problems, and the next section describes how our matheuristic defines \( \mathrm{D} \).

\subsection{RSSV Matheuristic}
\label{subsec:rssv}

Our RSSV matheuristic consists of six steps. Five of them follow the approach proposed by Mu \textit{et al.}~\citep{Mu2020}: Random sampling, Sub-problem solution, Spatial voting, Filtering, and Final problem solution. These are followed by a Post-optimization step. \FA{Overall, our heuristic combines ideas from Mu \textit{et al.}~\citep{Mu2020} with elements of the IRMA by Stefanello \textit{et al.}~\citep{stefanello2015matheuristics}.} Figure~\ref{fig:rssv_flowchart} in the Appendix shows the overall heuristic approach. A detailed explanation of each RSSV step is presented below.

\paragraph{ \textbf{Random Sampling:}}
\label{subsubsec:RSSV-sub_cplp_creation}
To handle large-scale instances, the original RSSV applies a random sampling to reduce both the set of customers and the set of locations, generating multiple smaller sub-problems to explore the solution space and identify promising candidate locations. 
Unlike \cite{Mu2020}, we retain the full customer set $I$ in the sub-problems. Each sub-problem is then treated as a $p$MP, incorporating territorial coverage when needed. Each is defined as \( \langle J_u, I, p, W, R \rangle \), where \( J_u \subseteq J \) is the subset of candidate locations in sub-problem \( u \in \{1, \dots, M\} \), and \( M \) is the number of sub-problems. We set \( |J_u|=n_{\text{cand}} \geq p \), for all \( u \in \{1, \dots, M\} \), and generate \( M \) sub-problems by randomly selecting \( n_{\text{cand}} \) locations from \( J \). Following the original work, \( M = \frac{5|J|}{n_{\text{cand}}} \), but now we limit it to \( M \leq 20 \), since we solve them in parallel on 20 threads. In our experiments, different values of \( n_{\text{cand}} \) are used.

\paragraph{ \textbf{Sub-Problem Solving:}}
\label{subsubsec:RSSV-sub_plp_solving}
Each sub-problem from the previous step is solved using the simple TB heuristic~\citep{teitz1968heuristic}. The heuristic starts from a randomly generated initial solution and applies a steepest descent search, iteratively improving the solution by replacing one selected facility location with a better alternative. The process stops when no single swap improves the solution. For the C$p$LP with coverage constraints, both the initial solution and the search respect these constraints, allowing swaps only if they maintain the required subarea coverage. A time limit of 5\% of the total instance limit is imposed for each sub-problem. Since sub-problems are solved in parallel, the total time remains within 5\% of the overall limit.


\paragraph{ \textbf{Spatial Voting:}}
\label{subsubsec:RSSV-voting}
Once all sub-problems have been solved, this step identifies a set of the most promising locations for the overall instance. To do so, each location is assigned a score based on two criteria: its frequency of appearance in sub-problem solutions, and its proximity to other frequently selected locations. Locations with higher voting scores are considered more likely to be part of the set of best solutions. The goal is to prioritize locations that are both commonly chosen and spatially close to other high-frequency ones, thereby reinforcing spatial consistency in the final selection.

Let $f_u = \{f_{u,1}, f_{u,2}, ..., f_{u,p}\}$ represent the set of $p$ locations selected in the solution of the $u^\text{th}$ sub-problem $\langle J_u, I, p, W, R \rangle$. Each candidate location $j \in J$ in the original problem is influenced by each location $f_{u,g} \in f_u$, for $u \in \{1,2,...,M\}$ and $g \in \{1,2,...,p\}$, with a score defined as:
\begin{equation*}
\delta_{jug}  = 
\begin{cases} 
e^{-\frac{\text{dist}^2(f_{u,g}, j)}{(\kappa h)^2}}, & \text{if } 0 < \text{dist}(f_{u,g}, j) \leq \kappa h,\\
0, & \text{otherwise.}
\end{cases}
\end{equation*}

The value $\text{dist}(f_{u,g}, j)$ is the distance between locations $f_{u,g}$ and $j$.
\citet{Mu2020} 
use $\kappa = 1$ and compute $h$ using Silverman's rule of thumb~\citep{silverman1998density}, which determines the bandwidth controlling spatial influence. It is defined as $
h = \left( \frac{4 \sigma^5}{3 |N|} \right)^{0.2},
$
where $\sigma$ is the standard deviation of the distances in the instance. In Section~\ref{sec:Analysis_Params_Strengthening}, we analyze how the value of $\kappa$ relates to the instance being solved.
Finally, the voting score function $V: J \to \mathbb{R}_+$, which represents the accumulated score for each location $j \in J$, is defined as
$
V(j) = \sum_{u=1}^{M} \sum_{g=1}^{p} \delta_{jug}.
$

\paragraph{ \textbf{Filtering:}}
\label{subsubsec:RSSV-filtering}
After computing the voting weights for all \( |J| \) locations, they are sorted in descending order, and the top \( n_{\text{cand}} \)  locations are retained as candidates for the next phase. We set a minimum value of \( n_{\text{cand}} \geq 2p \), ensuring a sufficiently large and diverse sample for the subsequent optimization step. When fewer than  \( n_{\text{cand}} \)  locations have strictly positive voting weights, additional sites are selected based on the largest capacities to increase the chances of obtaining a feasible solution.

\paragraph{ \textbf{Final Problem Solving:}}
\label{subsubsec:RSSV-finalproblem_solving}
Let \( J_{n_{\text{cand}}} \subseteq J \) denote the set of the \( n_{\text{cand}} \) most voted locations selected in the previous step. In this work, we solve the reduced problem instance \( \langle J_{n_{\text{cand}}}, I, p, W, R \rangle \) using the ILP formulation described in Section~\ref{sec:ProbDef} for the problem being solved, enriched with the valid inequalities presented in the previous section. Remember that the distance constraints~\eqref{cPMP_SatifiedCustomer_Dmax} require the definition of a maximum distance value \( \mathrm{D} \). We defined \( \mathrm{D} \) as the minimum of the maximum distances observed in the sub-problem solutions obtained during the second step of the RSSV heuristic.

In Section~\ref{sec:Analysis_Params_Strengthening}, we discuss how to efficiently combine the inequalities presented in  Section~\ref{subsec:form_strength} when solving the C\(p\)LP, followed by the impact of our choice of \( \mathrm{D} \) on the quality of the solution (in Section~\ref{sec:CompExp_RSSV_paca_coverage}).
The set of locations in the final solution obtained is denoted as \( f^{*} = \{f^{*}_{1}, f^{*}_{2}, ..., f^{*}_{p}\} \).

\paragraph{ \textbf{Post-Optimization:}}
If the RSSV Final Problem Solving step finishes before the time limit, the remaining time is used for a post-optimization step improving the best-found $f^{*}$. \textcolor{black}{The idea is to locally refine the neighborhood of $f^{*}$, exploring solutions close to it. This enables a chained effect of improvements, where a local modification can trigger further beneficial changes in related parts of the solution, leading to a broader exploration of the surrounding solution space.} A reduced candidate set $J'$ is built and explored with an ILP formulation.

We start with $J' = f^{*}$. For each location $f_g \in f^{*}$, we add the closest neighbor $j' \in J \setminus J'$, ensuring no neighbor is repeated, and forming $|J'| = 2p$. The problem $\langle J', I, p, W, R \rangle$ is then re-solved with ILP, using \( \mathrm{D} \) equal to the maximum customer–facility distance in the current $f^{*}$ solution.
If a better solution is found and time allows, the process restarts with the new $f^{*}$. Otherwise, the neighborhood is progressively expanded by including the next closest neighbors, until improvement is obtained. Possibly, if $|J'| = |J|$, a final ILP run is executed before ending the phase (see Algorithm~\ref{alg:post_optimization} in Appendix~\ref{sec:appendix-figures}).
\FA{The inclusion of this step is mainly motivated by instances with small subproblem sizes, for which the heuristic typically terminates significantly before the proposed time limit, leaving extra time available to further improve the solution.}



\section{Computational experiments}
\label{sec:CompExp}

The computational experiments address two main questions:

\begin{enumerate}
    \item Is the RSSV a computationally efficient approach for solving the variations of the C$p$LP?
    \item How do territorial coverage constraints in C$p$MP$^r$ affect the experiments and the resulting solutions in a real-world scenario?
\end{enumerate}

Section~\ref{sec:CompExp-Instances} describes the benchmark instances, including the construction of the PACA dataset, based on the French region \textit{Provence-Alpes-Côte d’Azur} (PACA). {\color{black}Experiments on literature instances (Section~\ref{sec:Exp_Literature-vs-RSSV}) and comparisons with ArcGIS Pro (Section~\ref{sec:Exp_ArcGIS-vs-MIP-vs-RSSV}) are detailed in~\ref{sec:appendix-experiments}}. 
Section~\ref{sec:Analysis_Params_Strengthening} analyzes the formulation-strengthening techniques and the impact of RSSV parameters ($n_{cand}$ and $\kappa$), leading to the best RSSV configuration for solving problems with territorial and multi-scale coverage constraints (Section~\ref{sec:CompExp_RSSV_paca_coverage}). Finally, Section~\ref{sec:Analysis_Solutions_PACA} presents an analysis of solutions for the PACA case study.

In the following experiments, the percentage gap (Gap [\%]) is calculated as
\(
\left( \frac{Z - Z^*}{Z^*} \right) \times 100,
\)
where \( Z \) is the value of the evaluated solution and \( Z^* \) the best known solution. 
GapILP [\%] is the gap reported by CPLEX at the time limit. 
All experiments, ran on an Intel Xeon E5640 CPU (2.67 GHz) with 128/256~GB RAM on Ubuntu 18.04.5 LTS. ILP and heuristic implementations were coded in C++, with formulations solved by IBM CPLEX 22.1.1. The RSSV results report the best of five runs, with the RSSV parameters summarized in Table~\ref{tab:rssv_params} and discussed in Section~\ref{sec:Analysis_Params_Strengthening}.
{\color{black}The RSSV method, the instance set, and the instance generation procedure are publicly available.\footnote{The github link will be made available here once the paper is accepted for publication.}.}

\subsection{Instances}
\label{sec:CompExp-Instances}


\subsubsection{Literature dataset:}

We evaluated the proposed matheuristic on three benchmark datasets widely used in the literature for C$p$MP. Although our objective is not to propose a better method for C$p$MP, we aim to verify that our method remains competitive on the problem without territorial coverage constraints.
The first dataset consists of five instances (p3038\_600 to p3038\_1000) proposed by 
\citet{lorena2003local},
derived from the TSPLIB
dataset and containing 3038 nodes,
with $p$ values ranging from 600 to 1000. The second dataset (fnl4461\_20 to fnl4461\_1000), introduced by 
\citet{stefanello2015matheuristics},
is also derived from TSPLIB and includes 4461 nodes, with $p$ varying between 20 and 1000. The third dataset, known as SJC, was proposed by \citet{Lorena2004} and is based on real-world data from Brazil. It contains up to 402 nodes and six instances with varying numbers of candidate locations and values of $p$, as detailed in Table~\ref{tab:rssv_cpmp_SJC}. This dataset was also used by \citet{stefanello2015matheuristics} and here serves to illustrate the behavior of our method on smaller problem sizes.
\textcolor{black}{Best known solutions for all instances are reported by \citet{stefanello2015matheuristics}.
\citet{gnagi2021matheuristic} provide an open-source implementation of their heuristic, along with their own implementation of IRMA \citep{stefanello2015matheuristics}. 
We compare against the better of the two on the smaller SJC instances. On the larger instances, our method, by adjusting only the $n_{\text{cand}}$ parameter, achieves better gaps than both methods in \citet{gnagi2021matheuristic}.
We further compared the 3038 and 4461 nodes instances against the more recent work of 
\citet{gjergji2026large}, where our method was competitive but did not achieve the best results.}

\textcolor{black}{To summarize, our method is competitive with recent C$p$MP heuristics \citep{gnagi2021matheuristic,gjergji2026large}: less effective on smaller instances, but competitive on larger ones by adjusting only $n_{\text{cand}}$, answering research question 1 (see~\ref{sec:Exp_Literature-vs-RSSV}).}



\subsubsection{PACA dataset:}

This study introduces two new sets of real-world instances, \texttt{paca5282} and \texttt{paca2641}, representing different spatial resolutions of the \textit{Provence-Alpes-Côte d’Azur} (PACA) region in southeastern France. The region has about 5 million inhabitants and exhibits strong contrasts in population density, from the major coastal cities and the urbanized littoral to sparsely populated rural and mountain areas. Both instances cover the contiguous continental part of the PACA region and are structured around three levels of territorial division: \emph{communes}, \emph{cantons}, and \textit{établissements publics de coopération intercommunale} (\emph{EPCIs}). \emph{Communes} and \emph{cantons} are standard administrative units, while \emph{EPCIs} are cooperative governance bodies grouping several communes for shared service provision and planning.

The PACA spatial structure was generated from a regular 2~km grid (7898 points). Demand at each grid point was estimated from the INSEE\footnote{\label{note:insee}\url{www.insee.fr/en/accueil}} 2021 gridded population dataset, which provides population counts at 1~km$^2$ resolution. For each 2~km grid cell, we aggregated surrounding 1~km$^2$ population values weighted by Euclidean distance to the cell center. Points with zero population were excluded (–2622), and six additional points (+6) were inserted to ensure that each \emph{commune}, \emph{canton}, and \emph{EPCI} contained at least one representative (with a demand value of one). The resulting 5282 grid points were used as both demand nodes and candidate facility sites, defining the \texttt{paca5282} instance. The travel times, distances between all pairs of nodes, were computed with the Open Source Routing Machine\footnote{\url{project-osrm.org}} (OSRM) using OpenStreetMap data, generating the final symmetric origin–destination matrices.

As noted in Section~\ref{sec:Intro}, the service analyzed is cinema. Due to the characteristics of this service, for the PACA instances we concentrate the experiments on the C$p$MP$^r$ variations. Although privately managed, cinemas in France operate under public concession and play a key cultural role. Using data from the 2023 \emph{Base Permanente des Équipements}\footref{note:insee} (BPE), we identified 192 cinemas in the PACA region and estimated their service capacities. Voronoi diagrams centered on each cinema were used to assign population demand within each cell. The resulting demand values were grouped into five capacity intervals, from which facility capacities were sampled and assigned to 2~km grid sites, preserving the observed spatial heterogeneity. The total number of cinemas (192) also served as the reference number of facilities $p$ in the instance. We used five values of \(p\in \{134, 173, 192, 211, 250\}\), representing \(\pm 10\%\) and \(\pm 30\%\) variations of the real number of cinemas identified.

The \texttt{paca5282} dataset thus contains 5282 valid points within the PACA region, covering 959 \emph{communes}, 192 \emph{cantons}, and 51 \emph{EPCIs}. To create the reduced \texttt{paca2641} dataset, one point was selected from each pair of neighboring grid cells while maintaining representation across all three territorial levels. 
The two PACA datasets are defined over these three territorial divisions. Territorial coverage instances treat each division independently, while multi-scale ones simultaneously consider \emph{communes} and \emph{EPCIs}, following Definition~\ref{def:multiscalePartition}. Combining different values of $p$, coverage types, and territorial layers generated 25 instances per dataset (\texttt{paca5282} and \texttt{paca2641}), providing realistic and diverse benchmarks for evaluating our methods with and without territorial coverage constraints. Figure~\ref{fig:paca2641_population_distribution} (\ref{sec:appendix-figures}) shows \texttt{paca2641} nodes and population distribution.
A natural question is how RSSV compares to commercial tools. We include a comparison with ArcGIS Pro in~\ref{sec:Exp_ArcGIS-vs-MIP-vs-RSSV}. The experiments on these instances will address research questions 1 and 2, evaluating RSSV effectiveness for the C$p$LP, C$p$LP-TC and C$p$LP-MTC.



\subsection{Analysis of Formulation Strengthening and Parameters}
\label{sec:Analysis_Params_Strengthening}

The following analyses focus on identifying the most effective RSSV configurations for the PACA instances.
The analyses comprise two complementary studies: first, assessing the influence of key parameters ($\kappa$ and $n_{\text{cand}}$) on solution quality (Section~\ref{sec:Analysis_Params_Strengthening-params}), and second, evaluating the effect of the formulation strengthening strategies introduced in Section~\ref{subsec:form_strength} (Section~\ref{sec:Analysis_Params_Strengthening-cuts}).

\subsubsection{Study of Formulation Tightening.}
\label{sec:Analysis_Params_Strengthening-cuts}

We evaluated the impact of the formulation-strengthening techniques described in Section~\ref{subsec:form_strength} on the \texttt{paca2641} instance, considering the variants C$p$MP$^r$, C$p$MP$^r$-TC, and C$p$MP$^r$-MTC. Three classes of valid inequalities were tested,

\begin{itemize}
    \item \textit{DisaggCuts} (Eq. \eqref{eq:valid_ineq_disagg}): For each facility $j \in J$ only to its closest weighted customer $i \in I$.
    
    \item Lifted $p$-cover inequalities (Eq. \eqref{eq:valid_ineq_lift_pcover}): We generated these by randomly selecting $p$ candidate locations and applying the lifting strategy. To balance coverage and model size, we included at most 250 such inequalities.
    
    \item Location upper bound constraints (Eq. \eqref{eq:valid_ineq_upper_limit_zone_set}): Applied to subsets $K' \subseteq K$ where $|K'| \in \{1, m^s - 1\}$, enforcing limits on small and large subarea groupings.
\end{itemize}

The combination of lifted $p$-cover and location upper-bound constraints is referred to as \textit{PropCuts}, while all three classes together form \textit{AllCuts}. Tables~\ref{tab:ilp_strength_paca2641} and~\ref{tab:ilp_strength_paca2641_cover} report their effects on solution quality for C$p$MP$^r$, C$p$MP$^r$-TC, and C$p$MP$^r$-MTC. Each table compares the baseline ILP formulation (Section~\ref{sec:ProbDef}) with strengthened versions using \textit{DisaggCuts}, \textit{PropCuts}, and \textit{AllCuts}. Reported metrics include Gap [\%] and GapILP [\%], while the “\textbf{\#Best}” row indicates how many instances achieved the lowest gap. All tests were performed with a 3600 seconds time limit, and the columns ``\textbf{ILP + \(<\text{cut}>\)}'' specify the additional constraints were applied.

The results for C$p$MP$^r$ (Table~\ref{tab:ilp_strength_paca2641}) and for the territorial coverage variants (Table~\ref{tab:ilp_strength_paca2641_cover}) show that \textit{DisaggCuts}, proposed in the literature, significantly improved performance on the \texttt{paca2641} instance, resulting in a lower average  Gap [\%] compared to the pure ILP formulation. This improvement is especially noticeable for problems without territorial coverage constraints.

Focusing on the problems with territorial coverage (Table~\ref{tab:ilp_strength_paca2641_cover}), the \textit{PropCuts} improved the ILP performance in some cases, particularly when the number of subareas to cover approaches $p$, as observed for the \emph{canton}-level coverage. The best overall results were obtained when combining all three classes of inequalities (\textit{AllCuts}), which achieve the lowest average gaps (Gap [\%] and GapILP [\%]) and the highest number of best solutions (\#Best).

\FA{Notably, while DisaggCuts exhibit the strongest individual performance, the combined AllCuts strategy consistently achieves the best overall results across the full benchmark. In particular, it yields lower optimality gaps for the majority of instances and, for some cases, reduces the gap dramatically (e.g., from approximately 45\% to around 3\%).}
Therefore, we recommend incorporating these additional constraints in order to tighten the formulation and improve the quality of the solution. For \texttt{paca2641}, \textit{DisaggCuts} were used for problems without territorial coverage, while \textit{AllCuts} were adopted for territorial and multiscale variants. The same configuration was applied to the \texttt{paca5282} instance, with one-hour time limit.

\begin{table}[htbp]
\centering
\caption{Effect of Formulation Strengthening on C$p$MP$^r$ for \texttt{paca2641}}
\label{tab:ilp_strength_paca2641}
{\tiny

\begin{tabular}{
    l 
    S[table-format=4.0]
    S[table-format=4.0]
    S[table-format=3.2, detect-weight, tight-spacing=true]
    S[table-format=3.2, detect-weight, tight-spacing=true]
    S[table-format=3.2, detect-weight, tight-spacing=true]
    S[table-format=3.2, detect-weight, tight-spacing=true]
}
\toprule
\textbf{Instance} & \(\boldsymbol{|J| = |I|}\) & \(\boldsymbol{p}\) &
\multicolumn{2}{c}{\textbf{ILP}} &
\multicolumn{2}{c}{\textbf{ILP + \textit{DisaggCuts}}} \\
\cmidrule(lr){4-5} \cmidrule(lr){6-7}
& & & {Gap [\%]} & {GapILP [\%]} & {Gap [\%]} & {GapILP [\%]} \\
\hline
\texttt{paca2641} & 2641 & 134 & 95.83 & 51.96 & 153.98 & 61.95 \\
\texttt{paca2641} & 2641 & 173 & 15.12 & 19.60 & \bfseries 9.42 & 11.89 \\
\texttt{paca2641} & 2641 & 192 & 18.88 & 18.91 & \bfseries 1.10 & 3.89 \\
\texttt{paca2641} & 2641 & 211 & 16.32 & 15.30 & \bfseries 1.42 & 3.45 \\
\texttt{paca2641} & 2641 & 250 & 18.05 & 16.32 & \bfseries 2.58 & 3.54 \\
\hline
\textbf{Average} & & & \bfseries 32.78 & 24.42 &  33.70 & \bfseries 16.94 \\
\textbf{\#Best}  & & & 1     & 1     & \bfseries 4               & \bfseries 4      \\
\bottomrule
\end{tabular}

}
\end{table}

\begin{table}[htbp]
\centering
\caption{Effect of Formulation Strengthening on C$p$MP$^r$-TC and C$p$MP$^r$-MTC for \texttt{paca2641}}
\label{tab:ilp_strength_paca2641_cover}
{\tiny
\resizebox{\textwidth}{!}{%

\begin{tabular}{
    l  
    S[table-format=4.0]  
    S[table-format=4.0]  
    l  
    l 
    S[table-format=2.2, detect-weight]  
    S[table-format=2.2, detect-weight]  
    S[table-format=2.2, detect-weight]  
    S[table-format=2.2, detect-weight]  
    S[table-format=2.2, detect-weight]  
    S[table-format=2.2, detect-weight]  
    S[table-format=2.2, detect-weight]  
    S[table-format=2.2, detect-weight]  
}
\toprule
\textbf{Instance} & \(\boldsymbol{|J| = |I|}\) & \(\boldsymbol{p}\) & \textbf{Territorial Div.} & \#\textbf{units} &
\multicolumn{2}{c}{\textbf{ILP}} &
\multicolumn{2}{c}{\textbf{ILP + \textit{DisaggCuts}}} &
\multicolumn{2}{c}{\textbf{ILP + \textit{PropCuts}}} &
\multicolumn{2}{c}{\textbf{ILP + \textit{AllCuts}}} \\
\cmidrule(lr){6-7} \cmidrule(lr){8-9} \cmidrule(lr){10-11} \cmidrule(lr){12-13}
& & & & & {Gap [\%]} & {GapILP [\%]} & {Gap [\%]} & {GapILP [\%]} & {Gap [\%]} & {GapILP [\%]} & {Gap [\%]} & {GapILP [\%]} \\
\hline
 \texttt{paca2641} & 2641 & 134 & EPCIs    &   51 & 69.10 & 43.49 & 42.34 & 32.03 & \bfseries 30.19 & 28.31 & 44.23 & 32.92 \\
 \texttt{paca2641} & 2641 & 173 & EPCIs    &   51 & 20.41 & 20.87 & 4.48  & 7.63  &  \bfseries 3.99 & 7.27 & 37.79 & 29.44 \\
 \texttt{paca2641} & 2641 & 192 & EPCIs    &   51 & \bfseries 18.53 & 18.14 & 21.01 & 19.13 & 19.24 & 19.22 & 20.56 & 19.30 \\
 \texttt{paca2641} & 2641 & 211 & EPCIs    &   51 & 22.66 & 20.31 & 3.37  & 5.34  & 14.86 & 8.21 & \bfseries 1.41 & 3.27 \\
 \texttt{paca2641} & 2641 & 250 & EPCIs    &   51 & 9.79  & 5.02  & \bfseries 0.64  & 1.65  & 11.99 & 12.33 & 3.20 & 4.21 \\
\hline
 \texttt{paca2641} & 2641 & 134 & \emph{cantons}  & 192 & 32.10 & 25.29 & \bfseries 2.39  & 3.85  & 8.37 & 9.41 & 2.58 & 3.87 \\
 \texttt{paca2641} & 2641 & 173 & \emph{cantons}  & 192 & \bfseries 1.33  & 1.83 & 2.25  & 2.57  & 1.90 & 2.35 & 1.93 & 2.26 \\
 \texttt{paca2641} & 2641 & 192 & \emph{cantons}  & 192 & 3.28  & 3.41 & 2.16  & 2.34  & \bfseries 0.76 & 1.05 & 7.98 & 7.61 \\
 \texttt{paca2641} & 2641 & 211 & \emph{cantons}  & 192 & 7.61  & 7.55 & 6.69  & 6.73  &\bfseries 0.47 & 0.97 & 0.76 & 3.08 \\
 \texttt{paca2641} & 2641 & 250 & \emph{cantons}  & 192 & 2.52  & 2.88 & 0.54  & 0.83  &   0.49 & 0.87 & \bfseries 0.40 & 0.84 \\
\hline
 \texttt{paca2641} & 2641 & 134 & \emph{communes} & 959 & 25.07 & 24.45 & \bfseries 3.45  & 5.68  & 8.56 & 10.87 & 28.02 & 23.77 \\
 \texttt{paca2641} & 2641 & 173 & \emph{communes} & 959 & \bfseries 16.51 & 16.23 & 35.39 & 26.97 & 52.55 & 35.48 & 29.28 & 23.74 \\
 \texttt{paca2641} & 2641 & 192 & \emph{communes} & 959 & 9.25  & 9.61 & 45.69 & 32.09 & 11.61 & 11.78 & \bfseries 2.91 & 3.75 \\
 \texttt{paca2641} & 2641 & 211 & \emph{communes} & 959 & 60.46 & 38.15 & 3.26  & 4.29  & 51.85 & 34.63 & \bfseries 1.21 & 2.20 \\
 \texttt{paca2641} & 2641 & 250 & \emph{communes} & 959 & 1.11  & 1.56 & 1.01  & 1.53  & 47.37 & 32.33 & \bfseries 0.69 & 1.10 \\
\hline
 \texttt{paca2641} & 2641 & 134 & EPCIs/\emph{communes} & 51/959 & 18.64 & 18.49 & 18.60 & 17.82 & 44.98 & 32.97 & \bfseries 2.33 &   4.74 \\
 \texttt{paca2641} & 2641 & 173 & EPCIs/\emph{communes}  & 51/959 & 8.12 & 9.62 & 4.63 & 5.83 & 10.27 & 11.09 & \bfseries 3.61 & 4.89 \\
 \texttt{paca2641} & 2641 & 192 & EPCIs/\emph{communes}  & 51/959 & 7.11  & 8.31 & 6.27 & 7.25 & 4.57 & 5.79 & \bfseries 2.89 & 3.89 \\
 \texttt{paca2641} & 2641 & 211 & EPCIs/\emph{communes}  & 51/959 & 44.62 & 31.35 & 2.57  & 3.34  & 34.20 & 26.26 & \bfseries 1.83 & 2.73 \\
 \texttt{paca2641} & 2641 & 250 & EPCIs/\emph{communes}  & 51/959 & 1.21  & 1.86 & \bfseries 0.38  & 0.90  & 0.69 & 1.15 & 0.47 & 1.06 \\
\hline
\textbf{Average} & & & & &
18.97 & 15.42 &
 10.36 &  9.39 &
17.95 & 14.62 &
\bfseries 9.70 & \bfseries 8.93 \\
\textbf{\#Best} & & & & &
3 & 3 &
4 & 4 &
4 & 4 &
\bfseries 9 & \bfseries 9 \\
\bottomrule
\end{tabular}

}
}
\end{table}

\subsubsection{Study of parameters.}
\label{sec:Analysis_Params_Strengthening-params}

{
Since our objective is to design a solution approach with as few tunable parameters as possible, we focus on systematically evaluating two key parameters of RSSV: the subproblem and final problem size ($n_{\text{cand}}$) and the bandwidth distance multiplier ($\kappa$).
We therefore analyzed how the objective function value behaved when running RSSV on the C$p$MP$^r$, with a one-hour time limit, for a fixed value of $p$ and different settings of $n_{\text{cand}}$ and $\kappa$.
The analyses are presented in Figures~\ref{fig:study_ncand_kappa_paca2641}--\ref{fig:study_spatial_kappa_paca2641}, using the instances \texttt{paca2641} and \texttt{paca5282}.
}

{
Figures~\ref{fig:study_ncand_kappa_paca2641} and~\ref{fig:study_ncand_kappa_paca5282} report the corresponding solution values obtained after running RSSV under different combinations of the parameters $\kappa \in \{1, 2, 3\}$ and candidate facility set sizes ($n_{\text{cand}} \in \{800, 1200, 1600, 2000\}$ for \texttt{paca2641}, and $n_{\text{cand}} \in \{1600, 2400, 3200, 4000\}$ for \texttt{paca5282}).
These figures reveal that both parameters have a noticeable impact on solution quality.
In general, increasing $n_{\text{cand}}$ tends to improve the objective value, as larger candidate sets enable the algorithm to explore more potential facility locations.
However, this idea may reverse for larger instances or smaller $p$ values, such as \texttt{paca5282} with $p=134$, where excessive candidate set sizes can complicate the solver’s ability to identify better quality solutions within the time limit.
In this case, the best results are obtained with a smaller candidate facility set ($n_{\text{cand}} = 2400$), highlighting that, for some instances, reducing the number of candidate locations can actually lead to better solutions under a limited execution time.
Regarding $\kappa$, higher values ($\kappa = 3$) lead to better solutions for \texttt{paca2641}, whereas smaller values ($\kappa = 1$) perform better for \texttt{paca5282}.
}


\begin{figure}[htbp]
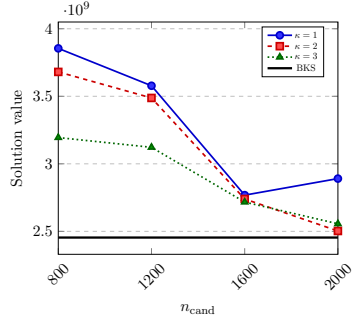

    \caption{Solution analysis for fixed $p$ with varying $\kappa$ and $n_\text{cand}$ for the \texttt{paca2641} instance.}
    \label{fig:study_ncand_kappa_paca2641}
    \centering
    \begin{subfigure}{0.48\textwidth}
        \centering
        \resizebox{0.6\linewidth}{!}{\input{figures/tikz_figs/surface_params_paca2641_p134}}
        \caption{$p=134$.}
        \label{fig:surface_p134}
    \end{subfigure}
    \hfill 
    \begin{subfigure}{0.48\textwidth}
        \centering
        \resizebox{0.6\linewidth}{!}{\input{figures/tikz_figs/surface_params_paca2641_p250}}
        \caption{$p=250$.}
        \label{fig:surface_p250}
    \end{subfigure}
    
\end{figure}

\begin{figure}[htbp]
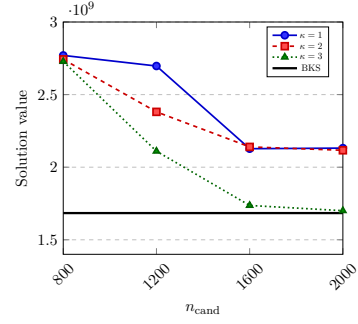

    \caption{Solution analysis for fixed $p$ with varying $\kappa$ and $n_\text{cand}$ for the \texttt{paca5282} instance.}
    \centering
    \begin{subfigure}{0.48\textwidth}
        \centering
        \resizebox{0.6\linewidth}{!}{\input{figures/tikz_figs/surface_params_paca_5282_p134}}
        \caption{$p=134$.}
        \label{fig:surface_p134_5282}
    \end{subfigure}
    \hfill 
    \begin{subfigure}{0.48\textwidth}
        \centering
        \resizebox{0.6\linewidth}{!}{\input{figures/tikz_figs/surface_params_paca_5282_p250}}
        \caption{$p=250$.}
        \label{fig:surface_p250_5282}
    \end{subfigure}
    
    \label{fig:study_ncand_kappa_paca5282}
\end{figure}

To understand the spatial effect of the parameter $\kappa$, we analyze the locations selected during the RSSV Filtering phase (Section~\ref{subsec:rssv}). Figure~\ref{fig:study_spatial_kappa_paca2641} shows the case of instance \texttt{paca2641} ($n_\text{cand} = 800$, $p = 134$) for different values of $\kappa \in \{1, 2, 3\}$. Blue dots indicate the selected locations prior to the Final Problem Solving step. Smaller $\kappa$ (e.g., $\kappa = 1$) leads to a more spatially dispersed configuration, while larger $\kappa$ (e.g., $\kappa = 3$) results in clustering, as the selection favors local density from the RSSV subproblems. The behavior for instance \texttt{paca5282} is similar (See Figure~\ref{fig:study_spatial_kappa_paca5282} in~\ref{sec:appendix-figures}). Note that for smaller candidate sets, increasing $\kappa$ too much can make the distance threshold constraints harder to satisfy, potentially causing infeasibility in the final step.

\begin{figure}[htbp]
    \caption{Spatial analysis for fixed $p=134$ and $n_{\text{cand}}=800$ with different $\kappa$ values for instance \texttt{paca2641}.}
    \centering
    \begin{subfigure}[t]{0.32\textwidth}
        \centering
        \includegraphics[width=1\linewidth]{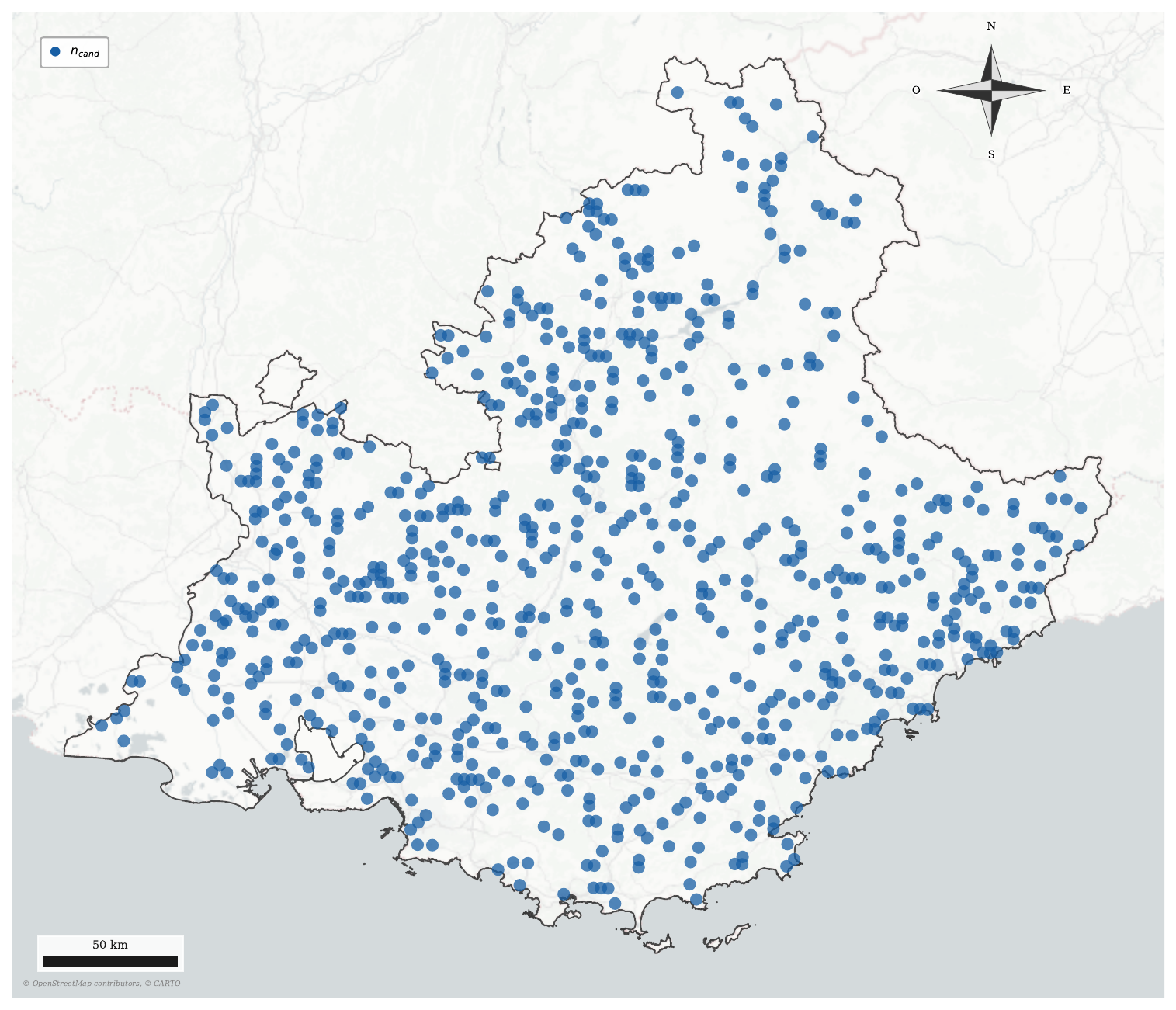}
        \caption{$\kappa = 1$}
        \label{fig:spatial_kappa_1}
    \end{subfigure}
    \hfill 
    \begin{subfigure}[t]{0.32\textwidth}
        \centering
        \includegraphics[width=\linewidth]{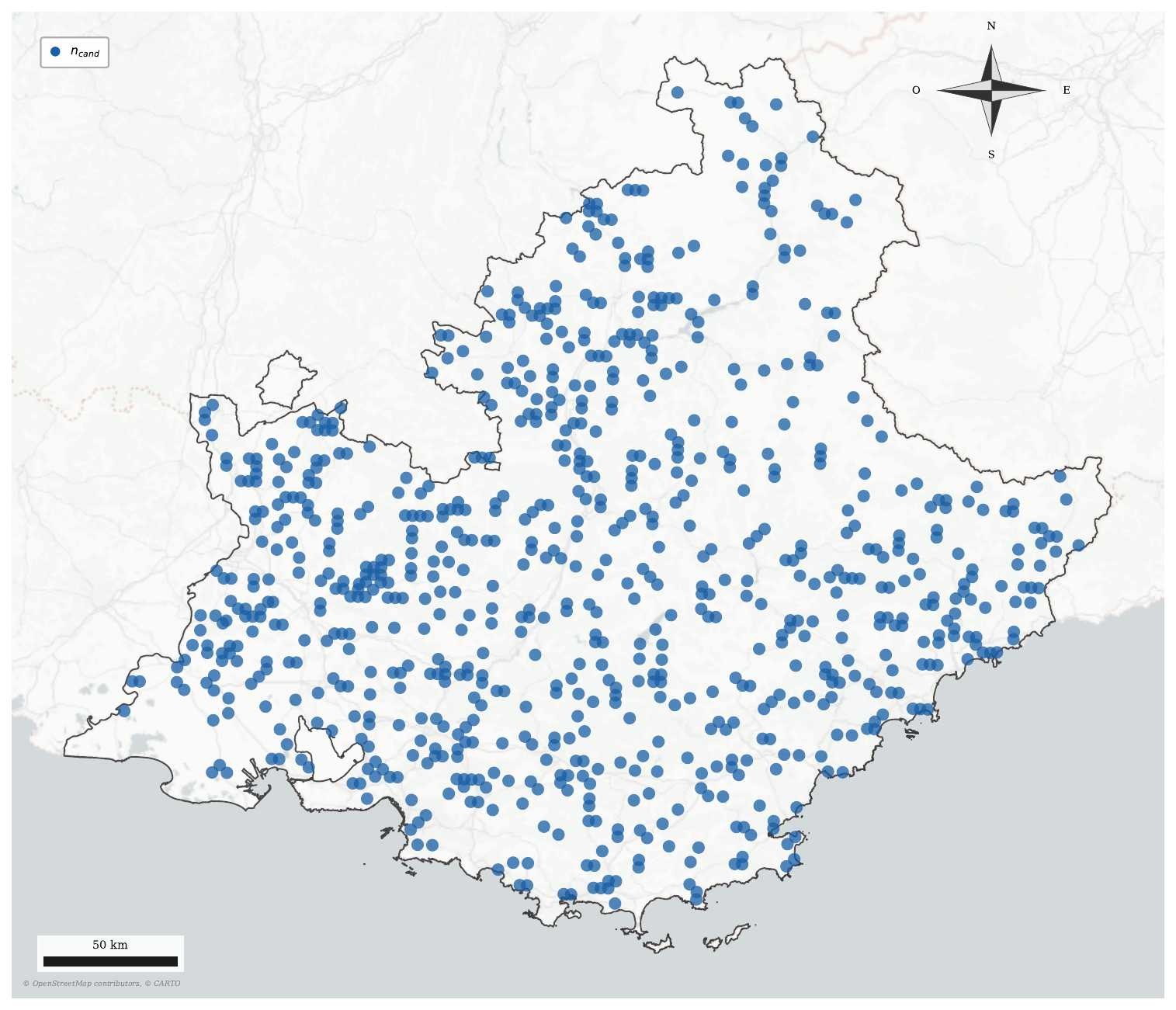}
        \caption{$\kappa = 2$}
        \label{fig:spatial_kappa_2}
    \end{subfigure}
    \hfill
    \begin{subfigure}[t]{0.32\textwidth}
        \centering
        \includegraphics[width=\linewidth]{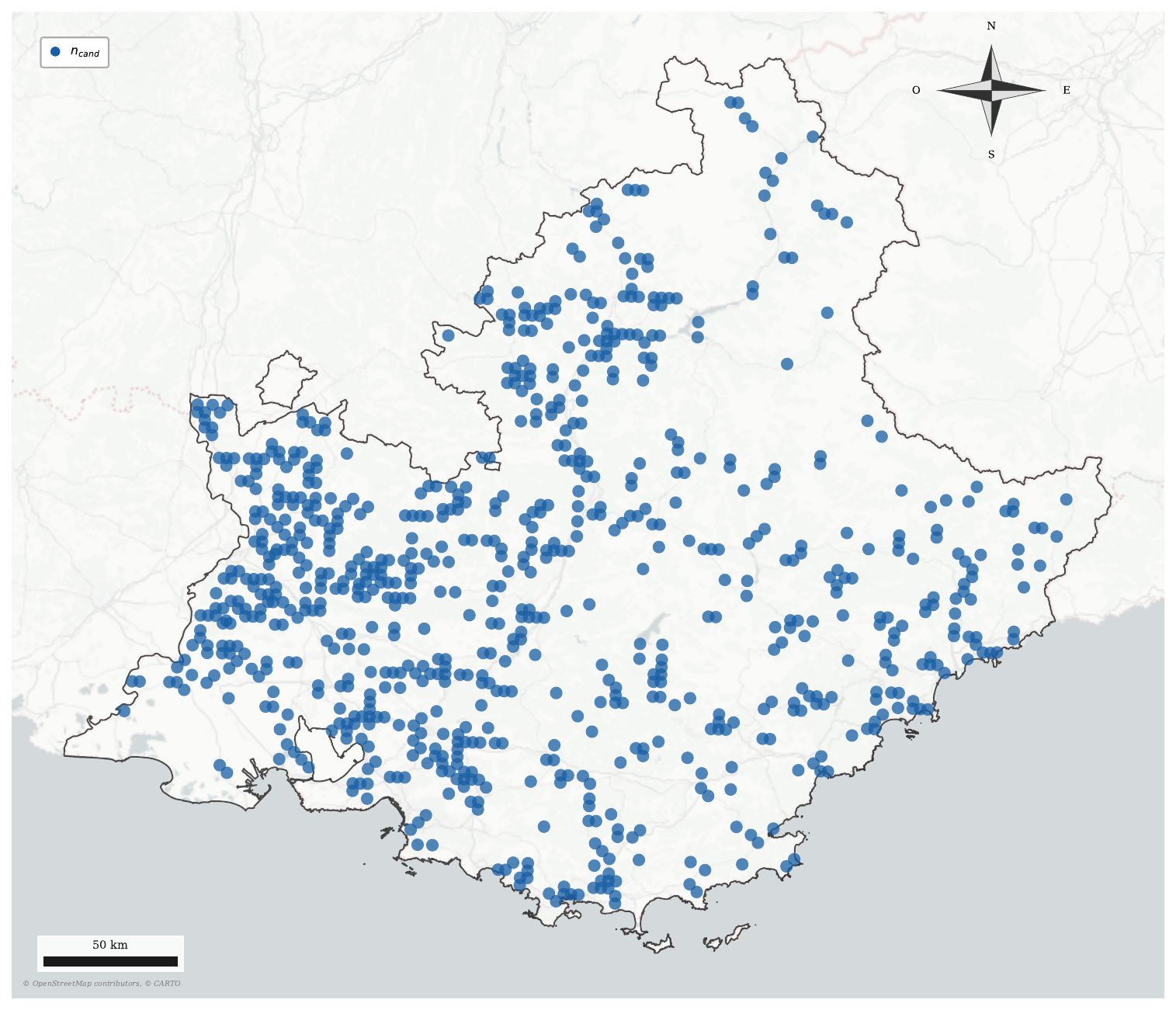}
        \caption{$\kappa = 3$}
        \label{fig:spatial_kappa_3}
    \end{subfigure}
    
    \label{fig:study_spatial_kappa_paca2641}
\end{figure}


\FA{In summary, based on these analyses, we recommend performing a sensitivity study of $\kappa$ and $n_{\text{cand}}$ for each new instance, in order to identify the configuration that best suits its specific spatial and demand characteristics.} For the instance \texttt{paca2641}, we set $n_{\text{cand}} = |J|$ and $\kappa = 3$. For the larger instance \texttt{paca5282}, we considered two candidate set sizes, $n_{\text{cand}} \in {2400, |J|}$, with $\kappa = 1$.

Table~\ref{tab:rssv_params} summarizes the RSSV parameters used in all experiments. Grouped into two categories: Parameters, including the time limit (\texttt{TL} [s]), $n_{\text{cand}}$, $\kappa$, and the number of subproblems ($M$), and Modes, which control options such as post-optimization and the inclusion of additional valid inequalities for ILP tightening. \texttt{DC} denotes the default configuration, while explicit values indicate modifications tested in specific experiments.

\begin{table}[htbp]
\centering
\caption{RSSV configuration parameters applied across computational experiments}
\label{tab:rssv_params}
{\tiny
\begin{tabular}{llcccccc}
\hline
\multicolumn{2}{c}{Configuration} & Default RSSV & Tables \ref{tab:rssv_paca2641} \& \ref{tab:rssv_paca2641_cover} & Tables \ref{tab:rssv_paca5282} \& \ref{tab:rssv_paca5282_cover} & Table \ref{tab:gap_comparison_arcgis} & Table \ref{tab:rssv_cpmp_lns_p3038_fnl4461} & Table \ref{tab:rssv_cpmp_SJC} \\ 
\hline
\multirow{4}{*}{Parameters} 
 & \texttt{TL} [s] & 3600 & \texttt{DC} & \texttt{DC} & 840 & \texttt{DC} & 40 \\ 
 & $n_{\text{cand}}$ & $|J|$ & \texttt{DC} & \{$|J|, 2400$\} & $|J|$ & $\{|J|, 2p + 0.1|J|\}$ & $\{|J|, |J|/2\}$ \\ 
 & $\kappa$ & 1 & 3 & \texttt{DC} & 1 & \texttt{DC} & \texttt{DC} \\ 
 & $M$ & $\frac{5|J|}{n_{\text{cand}}}$ & \texttt{DC} & \texttt{DC} & \texttt{DC} & \texttt{DC} & \texttt{DC} \\[3pt] 
 \toprule
\multirow{2}{*}{Modes} 
 & Post-Optimization & On & \texttt{DC} & \texttt{DC} & On & \texttt{DC} & Off \\ 
 & ILP Tightening & Off & \begin{tabular}[c]{@{}c@{}}\textit{DisaggCuts},\\ \textit{AllCuts} \end{tabular} & \begin{tabular}[c]{@{}c@{}}\textit{DisaggCuts},\\ \textit{AllCuts} \end{tabular} & \texttt{DC} & \texttt{DC} & \texttt{DC} \\ 
\hline
\multicolumn{8}{l}{\texttt{TL} [s]: Total time limit in seconds} \\
\multicolumn{8}{l}{\texttt{DC}: Same as default RSSV configuration} \\
\multicolumn{8}{l}{Tables \ref{tab:rssv_cpmp_lns_p3038_fnl4461}, \ref{tab:rssv_cpmp_SJC}, and \ref{tab:gap_comparison_arcgis} are reported in the Appendix \ref{sec:appendix-experiments}.} \\
\hline
\end{tabular}

}
\end{table}

\subsection{Comparing the RSSV Matheuristic and ILP with Territorial Coverage Constraints}
\label{sec:CompExp_RSSV_paca_coverage} 

We now asses the RSSV matheuristic on the \texttt{paca2641} and \texttt{paca5282} instances for solving the C$p$MP$^r$ with territorial coverage constraints. The results are summarized in four tables: for \texttt{paca2641}, Table~\ref{tab:rssv_paca2641} reports solutions without coverage constraints, while Table~\ref{tab:rssv_paca2641_cover} includes territorial (C$p$MP$^r$-TC) and multi-scale coverage (C$p$MP$^r$-MTC). Similarly, Tables~\ref{tab:rssv_paca5282} and~\ref{tab:rssv_paca5282_cover} present results for \texttt{paca5282}, without and with coverage constraints.

The columns labeled ``\textbf{RSSV($n_{\text{cand}}$) + \(<\text{cut}>\)}'' report RSSV results, where $n_{\text{cand}}$ is the number of candidate locations and \(<\text{cut}>\) indicates the type of constraints applied in the Final Problem Solving step. The attribute ``$\mathrm{D}$ [s]'' represents the distance threshold applied in the ILP or set by the RSSV heuristic. All other columns follow the format of previous tables. RSSV was executed using the configurations defined in Table \ref{tab:rssv_params}.

For \texttt{paca2641}, we let the ILP run for 5 hours with 256~GB of memory; the best-known solution used to compute Gap [\%] was then taken as the better of this ILP solution (gaps at most 1\%) and the RSSV matheuristic solution.
Tables~\ref{tab:rssv_paca2641} and~\ref{tab:rssv_paca2641_cover} report the best solutions from Tables~\ref{tab:ilp_strength_paca2641} and~\ref{tab:ilp_strength_paca2641_cover} as ``\textbf{ILP}\textsubscript{BestCuts}'' for direct comparison with ILP solutions strengthened by valid inequalities, highlighting the superior solution quality of RSSV. To assess the effect of the threshold $\mathrm{D}$ defined by the RSSV heuristic, the ILP was also solved with $\mathrm{D} = 7200$\,s (column ``\textbf{ILP ($\mathrm{D}$ = 7200\,s)}''), showing that RSSV's thresholds are flexible enough to explore good solutions.

\begin{table}[htbp]
\centering
\caption{Performance comparison of RSSV and ILP on instance \texttt{paca2641} for C$p$MP$^r$.}
\label{tab:rssv_paca2641}
{\tiny
\begin{tabular}{
    l  
    S[table-format=4.0]  
    S[table-format=3.0]  
    S[table-format=2.2, detect-weight]  
    l  
    S[table-format=2.2, detect-weight]  
    l  
    S[table-format=4.2]  
    S[table-format=2.2, detect-weight]  
    l  
    S[table-format=4.0]  
}
\toprule
\textbf{Instance} & \(\boldsymbol{|J| = |I|}\) & \(\boldsymbol{p}\) & 
\multicolumn{2}{c}{\textbf{ILP}\textsubscript{BestCuts}} & 
\multicolumn{3}{c}{\textbf{RSSV($|J|$) + \textit{DisaggCuts}}} & 
\multicolumn{3}{c}{\textbf{ILP ($\mathrm{D}$ = 7200\,s)}} \\
\cmidrule(lr){4-5} \cmidrule(lr){6-8} \cmidrule(lr){9-11}
& & & {Gap [\%]} & {Time [s]} & {Gap [\%]} & {Time [s]} & {$\mathrm{D}$ [s]} & {Gap [\%]} & {Time [s]} & {$\mathrm{D}$ [s]} \\
\hline
\texttt{paca2641} & 2641 & 134  & 95.83 & \texttt{TL} & \bfseries 0.13 & \texttt{TL} & 3656.7 & 23.54  & \texttt{TL} & 7200 \\
\texttt{paca2641} & 2641 & 173  & 9.42  & \texttt{TL} & \bfseries 0.00 & \texttt{TL} & 3544.2 & 12.75 & \texttt{TL} & 7200 \\
\texttt{paca2641} & 2641 & 192  & 1.10 & \texttt{TL} & \bfseries 0.01 & \texttt{TL} & 3609.1 & 8.20 & \texttt{TL} & 7200 \\
\texttt{paca2641} & 2641 & 211  & 1.42  & \texttt{TL} & \bfseries 0.00 & \texttt{TL} & 3958.0 & 9.46  & \texttt{TL} & 7200 \\
\texttt{paca2641} & 2641 & 250  & 2.58  & \texttt{TL} & \bfseries 0.00 & \texttt{TL} & 3470.2 & 0.05  & \texttt{TL} & 7200 \\
\hline
\textbf{Average} & & & 22.07 &  & 0.03 &  & 3609.1 & 10.80 &  & 7200 \\
\multicolumn{11}{l}{\texttt{TL}: Time limit of 3600 seconds} \\
\bottomrule
\end{tabular}

}
\end{table}

\begin{table}[htbp]
\centering
\caption{Performance comparison of RSSV and ILP on instance \texttt{paca2641} for C$p$MP$^r$-TC and C$p$MP$^r$-MTC.}
\label{tab:rssv_paca2641_cover}
{\tiny

\begin{tabular}{
    l  
    S[table-format=4.0]  
    S[table-format=4.0]  
    l  
    l  
    S[table-format=2.2, detect-weight]  
    l  
    S[table-format=2.2, detect-weight]  
    l  
    S[table-format=4.1]  
    S[table-format=2.2, detect-weight]  
    l  
    S[table-format=4.0]  
}
\toprule
\textbf{Instance} & \(\boldsymbol{|J| = |I|}\) & \(\boldsymbol{p}\) & \textbf{Territorial Div.} & \#\textbf{units} &
\multicolumn{2}{c}{\textbf{ILP}\textsubscript{BestCuts}} &
\multicolumn{3}{c}{\textbf{RSSV($|J|$) + \textit{AllCuts}}} &
\multicolumn{3}{c}{\textbf{ILP ($\mathrm{D}$ = 7200\,s)}} \\
\cmidrule(lr){6-7} \cmidrule(lr){8-10} \cmidrule(lr){11-13}
& & & & & {Gap [\%]} & {Time [s]} & {Gap [\%]} & {Time [s]} & {$\mathrm{D}$ [s]} & {Gap [\%]} & {Time [s]} & {$\mathrm{D}$ [s]} \\
\hline
\texttt{paca2641} & 2641 & 134 & EPCIs    &   51 & 30.19 & \texttt{TL} & {\bfseries 0.23} & \texttt{TL} & 3610.4 & 8.47  & \texttt{TL} & 7200 \\
\texttt{paca2641} & 2641 & 173 & EPCIs    &   51 & 3.99  & \texttt{TL} & {\bfseries 0.00} & \texttt{TL} & 3630.8 & 17.53 & \texttt{TL} & 7200 \\
\texttt{paca2641} & 2641 & 192 & EPCIs    &   51 & 18.53 & \texttt{TL} & {\bfseries 0.00} & \texttt{TL} & 3809.8 & 12.47 & \texttt{TL} & 7200 \\
\texttt{paca2641} & 2641 & 211 & EPCIs    &   51 & 1.41  & \texttt{TL} & {\bfseries 0.00} & \texttt{TL} & 3544.2 & 1.76  & \texttt{TL} & 7200 \\
\texttt{paca2641} & 2641 & 250 & EPCIs    &   51 & 0.64  & \texttt{TL} & {\bfseries 0.00} & \texttt{TL} & 3783.3 & 0.12  & \texttt{TL} & 7200 \\ \hline
\texttt{paca2641} & 2641 & 134 & \emph{cantons}  & 192 & 2.39  & \texttt{TL} & {\bfseries 0.00} & \texttt{TL} & 3769.2 & 0.52  & \texttt{TL} & 7200 \\ 
\texttt{paca2641} & 2641 & 173 & \emph{cantons}  & 192 & 1.33  & \texttt{TL} & {\bfseries 0.00} & \texttt{TL} & 3822.1 & 0.17  & \texttt{TL} & 7200 \\
\texttt{paca2641} & 2641 & 192 & \emph{cantons}  & 192 & 0.76  & \texttt{TL} & {\bfseries 0.00} & \texttt{TL} & 3810.3 & 0.15  & \texttt{TL} & 7200 \\
\texttt{paca2641} & 2641 & 211 & \emph{cantons}  & 192 & 0.47  & \texttt{TL} & {\bfseries 0.00} & \texttt{TL} & 3544.2 & 0.26  & \texttt{TL} & 7200 \\
\texttt{paca2641} & 2641 & 250 & \emph{cantons}  & 192 & 0.40  & \texttt{TL} & {\bfseries 0.00} & \texttt{TL} & 3735.8 & 0.08  & \texttt{TL} & 7200 \\ \hline
\texttt{paca2641} & 2641 & 134 & \emph{communes} & 959 & 3.45  & \texttt{TL} & {\bfseries 0.00} & \texttt{TL} & 3769.2 & 31.32 & \texttt{TL} & 7200 \\ 
\texttt{paca2641} & 2641 & 173 & \emph{communes} & 959 & 16.51 & \texttt{TL} & {\bfseries 0.00} & \texttt{TL} & 4297.9 & 1.23  & \texttt{TL} & 7200 \\
\texttt{paca2641} & 2641 & 192 & \emph{communes} & 959 & 2.91  & \texttt{TL} & {\bfseries 0.00} & \texttt{TL} & 4208.0 & 1.06  & \texttt{TL} & 7200 \\
\texttt{paca2641} & 2641 & 211 & \emph{communes} & 959 & 1.21  & \texttt{TL} & {\bfseries 0.00} & \texttt{TL} & 3822.1 & 0.10  & \texttt{TL} & 7200 \\
\texttt{paca2641} & 2641 & 250 & \emph{communes} & 959 & 0.69  & \texttt{TL} & {\bfseries 0.00} & \texttt{TL} & 3634.4 & 0.41  & \texttt{TL} & 7200 \\ \hline
\texttt{paca2641} & 2641 & 134 & EPCIs/\emph{communes} & 51/959 & 2.33 & \texttt{TL} & {\bfseries 0.06} & \texttt{TL} & 3920.3 & 6.04  & \texttt{TL} & 7200 \\ 
\texttt{paca2641} & 2641 & 173 & EPCIs/\emph{communes} & 51/959 & 3.61 & \texttt{TL} & {\bfseries 0.00} & \texttt{TL} & 4440.9 & 4.30  & \texttt{TL} & 7200 \\
\texttt{paca2641} & 2641 & 192 & EPCIs/\emph{communes} & 51/959 & 2.89 & \texttt{TL} & {\bfseries 0.00} & \texttt{TL} & 3879.1 & 0.82  & \texttt{TL} & 7200 \\
\texttt{paca2641} & 2641 & 211 & EPCIs/\emph{communes} & 51/959 & 1.83 & \texttt{TL} & {\bfseries 0.00} & \texttt{TL} & 3916.9 & 0.39  & \texttt{TL} & 7200 \\
\texttt{paca2641} & 2641 & 250 & EPCIs/\emph{communes} & 51/959 & 0.38 & \texttt{TL} & {\bfseries 0.00} & \texttt{TL} & 4760.6 & 0.10  & \texttt{TL} & 7200 \\ 
\hline
\textbf{Average} & & & & & 4.80 &  & {\bfseries 0.01} &  & 3885.5 & 4.36 &  & 7200 \\
\multicolumn{13}{l}{\texttt{TL}: Time limit of 3600 seconds} \\
\bottomrule
\end{tabular}

}
\end{table}

\begin{table}[htbp]
\centering
\caption{Performance comparison of RSSV and ILP on instance \texttt{paca5282} for C$p$MP$^r$.}
\label{tab:rssv_paca5282}
{\tiny
\begin{tabular}{
    l  
    S[table-format=4.0]  
    S[table-format=3.0]  
    S[table-format=3.2, detect-weight]  
    S[table-format=3.2, detect-weight]  
    S[table-format=3.2, detect-weight]  
    S[table-format=3.2, detect-weight]  
    S[table-format=4.1]  
    S[table-format=3.2, detect-weight]  
    S[table-format=3.2, detect-weight]  
    S[table-format=4.1]  
}
\toprule
\textbf{Instance} & \(\boldsymbol{|J| = |I|}\) & \(\boldsymbol{p}\) & 
\multicolumn{2}{c}{\textbf{ILP}\textsuperscript{5h}} & 
\multicolumn{3}{c}{\textbf{RSSV($|J|$) + \textit{DisaggCuts}}} &
\multicolumn{3}{c}{\textbf{RSSV(2400) + \textit{DisaggCuts}}} \\
\cmidrule(lr){4-5} \cmidrule(lr){6-8} \cmidrule(lr){9-11}
& & & {Gap [\%]} & {GapILP [\%]} & {Gap [\%]} & {GapILP\textsuperscript{*} [\%]} & {$\mathrm{D}$ [s]} & {Gap [\%]} & {GapILP\textsuperscript{*} [\%]} & {$\mathrm{D}$ [s]} \\
\hline
\texttt{paca5282} & 5282 & 134  & 49.83 & 68.69 & 32.50 & 50.06 & 4163.9 & \bfseries 0.00 & 10.07 & 4572.8 \\
\texttt{paca5282} & 5282 & 173  & \bfseries 0.00  & 49.96 & 6.57 & 35.48 & 3978.3 & 8.36 & 9.35 & 4713.6 \\
\texttt{paca5282} & 5282 & 192  & 55.24 & 67.99 & \bfseries 0.00 & 30.02 & 4210.0 & 8.86 & 5.58 & 4391.0 \\
\texttt{paca5282} & 5282 & 211  & 1.88  & 43.93 & \bfseries 0.00 & 26.76 & 4441.6 & 15.02 & 5.74 & 4236.3 \\
\texttt{paca5282} & 5282 & 250  & 1.24  & 43.23 & \bfseries 0.00 & 25.49 & 4708.6 & 23.00 & 2.77 & 4266.1 \\
\hline
\textbf{Average} & & & 21.64 & 54.76 & \bfseries 7.81 & 33.56 & 4300.48 & 11.05 & 6.70 & 4435.96 \\
\textbf{\#Best} & & & 1 &  & \bfseries 3 &  &  & 1 &  &  \\
\multicolumn{11}{l}{GapILP\textsuperscript{*}: CPLEX gap for the reduced problem} \\
\bottomrule
\end{tabular}

}
\end{table}

For the larger \texttt{paca5282} instance, solving the ILP with a one hour time limit required increasing the memory to 256~GB, and the results presented very poor gaps.
Thus, RSSV performance is compared only with the ILP solved with a 5-hour limit (``\textbf{ILP}\textsuperscript{5h}''). Tables~\ref{tab:rssv_paca5282} and~\ref{tab:rssv_paca5282_cover} report GapILP\textsuperscript{*} [\%], the optimality gap returned by CPLEX for the reduced problem under the respective $\mathrm{D}$. Two values of $n_{\text{cand}}$ were considered: $|J|$ and $2400$, as discussed in Section~\ref{sec:Analysis_Params_Strengthening}. Overall, RSSV consistently outperformed the ILP within 5 hours, achieving lower gaps and more best solutions. For smaller $p$, a reduced candidate set is more effective. Without territorial coverage, considering all candidates ($n_{\text{cand}} = |J|$) performs better, whereas with coverage constraints, the reduced set ($n_{\text{cand}} = 2400$) provides the best average Gap [\%] and the number of best solutions identified (\textbf{\#Best}).




\begin{table}[htbp]
\centering
\caption{Performance comparison of RSSV and ILP on instance \texttt{paca5282} for C$p$MP$^r$-TC and C$p$MP$^r$-MTC.}
\label{tab:rssv_paca5282_cover}
{\tiny
\resizebox{\textwidth}{!}{%
\begin{tabular}{
    l  
    S[table-format=4.0]  
    S[table-format=4.0]  
    l  
    l  
    S[table-format=3.2, detect-weight]  
    S[table-format=2.2, detect-weight]  
    S[table-format=3.2, detect-weight]  
    S[table-format=2.2, detect-weight]  
    S[table-format=4.1]  
    S[table-format=2.2, detect-weight]  
    S[table-format=2.2, detect-weight]  
    S[table-format=4.1]  
}
\toprule
\textbf{Instance} & \(\boldsymbol{|J| = |I|}\) & \(\boldsymbol{p}\) & \textbf{Territorial Div.} & \#\textbf{units} &
\multicolumn{2}{c}{\textbf{ILP}\textsuperscript{5h}} &
\multicolumn{3}{c}{\textbf{RSSV($|J|$) + \textit{AllCuts}}} &
\multicolumn{3}{c}{\textbf{RSSV(2400) + \textit{AllCuts}}} \\
\cmidrule(lr){6-7} \cmidrule(lr){8-10} \cmidrule(lr){11-13}
& & & & & {Gap [\%]} & {GapILP [\%]} & {Gap [\%]} & {GapILP\textsuperscript{*} [\%]} & {$\mathrm{D}$ [s]} & {Gap [\%]} & {GapILP\textsuperscript{*} [\%]} & {$\mathrm{D}$ [s]} \\
\hline
\texttt{paca5282} & 5282 & 134 & EPCIs    &   51 & 65.39 & 66.84 & 215.18 & 79.19 & 4738.7 & \bfseries 0.00  & 13.49 & 4391.0 \\
\texttt{paca5282} & 5282 & 173 & EPCIs    &   51 & 66.48 & 72.67 & \bfseries 0.00   & 33.51 & 4465.6 & 12.87 & 9.91  & 4391.0 \\
\texttt{paca5282} & 5282 & 192 & EPCIs    &   51 & 43.31 & 64.89 & \bfseries 0.00   & 35.02 & 4514.6 & 15.06 & 8.76  & 4056.5 \\
\texttt{paca5282} & 5282 & 211 & EPCIs    &   51 & \bfseries 0.00  & 45.84 & 2.98   & 32.20 & 4514.6 & 21.68 & 4.35  & 4056.5 \\
\texttt{paca5282} & 5282 & 250 & EPCIs    &   51 & 22.64 & 52.00 & \bfseries 0.00   & 27.74 & 3978.3 & 31.00 & 6.31  & 4056.5 \\ \hline
\texttt{paca5282} & 5282 & 134 & \emph{cantons}  & 192 & 53.09 & 47.16 & 11.98  & 24.77 & 4441.6 & \bfseries 0.00  & 4.87  & 4441.6 \\ 
\texttt{paca5282} & 5282 & 173 & \emph{cantons}  & 192 & 60.27 & 42.87 & 20.11  & 23.24 & 4441.6 & \bfseries 0.00  & 2.04  & 4163.9 \\
\texttt{paca5282} & 5282 & 192 & \emph{cantons}  & 192 & 59.29 & 41.27 & 10.76  & 15.49 & 4163.9 & \bfseries 0.00  & 1.00  & 4163.9 \\
\texttt{paca5282} & 5282 & 211 & \emph{cantons}  & 192 & 29.45 & 40.06 & 9.74   & 23.72 & 3978.3 & \bfseries 0.00  & 6.13  & 3978.3 \\
\texttt{paca5282} & 5282 & 250 & \emph{cantons}  & 192 & 58.77 & 48.21 & \bfseries 0.00   & 13.42 & 4210.0 & 10.36 & 1.02  & 3978.3 \\ \hline
\texttt{paca5282} & 5282 & 134 & \emph{communes} & 959 & 260.59 & 81.14 & 26.35  & 41.35 & 4441.6 & \bfseries 0.00  & 7.64  & 4086.0 \\ 
\texttt{paca5282} & 5282 & 173 & \emph{communes} & 959 & 44.09  & 51.32 & 102.04 & 61.35 & 4708.6 & \bfseries 0.00  & 3.46  & 4391.0 \\
\texttt{paca5282} & 5282 & 192 & \emph{communes} & 959 & 34.31  & 47.53 & 7.19   & 28.04 & 4071.3 & \bfseries 0.00  & 3.11  & 3729.6 \\
\texttt{paca5282} & 5282 & 211 & \emph{communes} & 959 & 22.07  & 42.58 & \bfseries 0.00   & 21.34 & 4392.5 & 2.54  & 2.56  & 3575.9 \\
\texttt{paca5282} & 5282 & 250 & \emph{communes} & 959 & 16.92  & 33.55 & \bfseries 0.00   & 15.79 & 3920.1 & 0.77  & 3.46  & 3586.7 \\ \hline
\texttt{paca5282} & 5282 & 134 & EPCIs/\emph{communes} & 51/959 & 182.20 & 76.07 & 26.01  & 41.64 & 4925.5 & \bfseries 0.00  & 5.57  & 4086.0 \\ 
\texttt{paca5282} & 5282 & 173 & EPCIs/\emph{communes} & 51/959 & 46.64  & 55.41 & 15.10  & 35.23 & 4802.1 & \bfseries 0.00  & 7.09  & 4460.6 \\
\texttt{paca5282} & 5282 & 192 & EPCIs/\emph{communes} & 51/959 & 35.16  & 46.79 & \bfseries 0.00   & 21.79 & 3978.3 & 1.28  & 2.02  & 4391.0 \\
\texttt{paca5282} & 5282 & 211 & EPCIs/\emph{communes} & 51/959 & 24.99  & 42.99 & \bfseries 0.00   & 21.69 & 3828.4 & 5.14  & 4.26  & 3978.3 \\
\texttt{paca5282} & 5282 & 250 & EPCIs/\emph{communes} & 51/959 & 21.36  & 37.83 & \bfseries 0.00   & 17.73 & 3692.8 & 5.28  & 0.26  & 3541.6 \\ 
\hline
\textbf{Average} & & & & & 57.35 & 51.85 & 22.37 & 30.71 & 4310.4 & \bfseries 5.30 & 4.87 & 4075.2 \\
\textbf{\#Best} & & & & & 1 &  & 9 &  &  & \bfseries 10 &  & \\
\multicolumn{13}{l}{GapILP\textsuperscript{*}: CPLEX gap for the reduced problem} \\
\bottomrule
\end{tabular}

}
}
\end{table}

\subsection{Analysis of Solutions with Territorial Coverage for \texttt{paca2641}}
\label{sec:Analysis_Solutions_PACA}

We now evaluate the impact of territorial coverage constraints on the solutions of the C$p$MP$^r$. For this analysis, we used the best-known solutions obtained for the \texttt{paca2641} instance, as near-optimal solutions allowed a more meaningful comparison than the \textcolor{black}{larger-scale} \texttt{paca5282} instance.

Without coverage constraints, solutions tend to favor high-demand areas, concentrating facilities in densely populated zones. Enforcing coverage improves spatial equity by ensuring service across all territorial units, but with higher solution costs. We quantify this increase using the Relative Increase in Cost (RIC):
$
\left( \frac{\text{Solution}_{\textit{cover}} - \text{Solution}_{\textit{nocover}}}{\text{Solution}_{\textit{nocover}}} \right) \times 100
$,
where $\text{Solution}_{\textit{cover}}$ and $\text{Solution}_{\textit{nocover}}$ are the objective values obtained for C$p$MP$^r$ with and without coverage, respectively.

Figure~\ref{fig:relative_cost_increase} shows $\mathrm{RIC}$ across different values of $p$, comparing solutions without coverage (black) to those with coverage at different territorial levels: \emph{EPCIs} (blue, 51 units), \emph{cantons} (orange, 192 units), and \emph{communes} (green, 959 units). At the \emph{EPCI} level, cost increases are minimal (below 1.1\%), as the number of facilities generally suffices to cover most units. For \emph{cantons}, costs peak at 85\% when $p$ equals the number of \emph{canton} units, reflecting the constraint forcing facilities into lower-demand regions; beyond this point, increasing $p$ reduces relative cost. \emph{Communes}-level coverage shows an increasing relative cost with $p$, reflecting the larger number of spatial units to cover.

\begin{figure}[htbp]
    \caption{$\mathrm{RIC}$ of territorial cover models vs. C$p$MP$^r$ (no cover) applied to cinemas with varying $p$ for \texttt{paca2641}.}
    \centering
    \resizebox{0.7\linewidth}{!}{\begin{tikzpicture}
\begin{axis}[
    width=0.84\linewidth,
    height=0.5\linewidth,
    xlabel={$p$ (number of cinemas to install)},
    ylabel={Relative Increase in Cost (\%)},
    legend style={at={(1.05,1)},
                  anchor=north west,
                  font=\footnotesize,
                  /tikz/mark size=2pt,
                  nodes={scale=0.9, transform shape}},
    xtick={134,173,192,211,250},
    ymin=0,
    grid=major,
]

\addplot[
    black,
    ultra thick
] coordinates {
(134,0.0)
(173,0.0)
(192,0.0)
(211,0.0)
(250,0.0)
};
\addlegendentry{No Cover ($y = 0$)}

\addplot[
    color=blue,
    mark=*,
    ultra thick
] coordinates {
(134,1.361)
(173,0.689)
(192,0.385)
(211,0.255)
(250,0.114)
};
\addlegendentry{EPCI}
\node at (axis cs:138,1.361) [anchor=south east, font=\footnotesize] { 1.4\% };
\node at (axis cs:177,0.689) [anchor=south east, font=\footnotesize] { 0.7\% };
\node at (axis cs:196,0.385) [anchor=south east, font=\footnotesize] { 0.4\% };
\node at (axis cs:215,0.255) [anchor=south east, font=\footnotesize] { 0.3\% };
\node at (axis cs:254,0.114) [anchor=south east, font=\footnotesize] { 0.1\% };

\addplot[
    color=orange,
    mark=*,
    ultra thick
] coordinates {
(134,54.154)
(173,74.746)
(192,85.367)
(211,23.778)
(250,11.802)
};
\addlegendentry{Canton}
\node at (axis cs:138,54.154) [anchor=south east, font=\footnotesize] { 54.2\% };
\node at (axis cs:177,74.746) [anchor=south east, font=\footnotesize] { 74.7\% };
\node at (axis cs:196,85.367) [anchor=south east, font=\footnotesize] { 85.4\% };
\node at (axis cs:215,23.778) [anchor=south east, font=\footnotesize] { 23.8\% };
\node at (axis cs:254,11.802) [anchor=south east, font=\footnotesize] { 11.8\% };

\addplot[
    color=green!70!black,
    mark=*,
    ultra thick
] coordinates {
(134,30.739)
(173,38.233)
(192,42.115)
(211,46.402)
(250,54.5)
};
\addlegendentry{Commune}
\node at (axis cs:138,31.051) [anchor=south east, font=\footnotesize] { 31.1\% };
\node at (axis cs:177,38.208) [anchor=south east, font=\footnotesize] { 38.2\% };
\node at (axis cs:196,42.115) [anchor=south east, font=\footnotesize] { 42.1\% };
\node at (axis cs:215,46.395) [anchor=south east, font=\footnotesize] { 46.4\% };
\node at (axis cs:254,54.499) [anchor=south east, font=\footnotesize] { 54.5\% };

\end{axis}
\end{tikzpicture}}
    
    \label{fig:relative_cost_increase}
\end{figure}
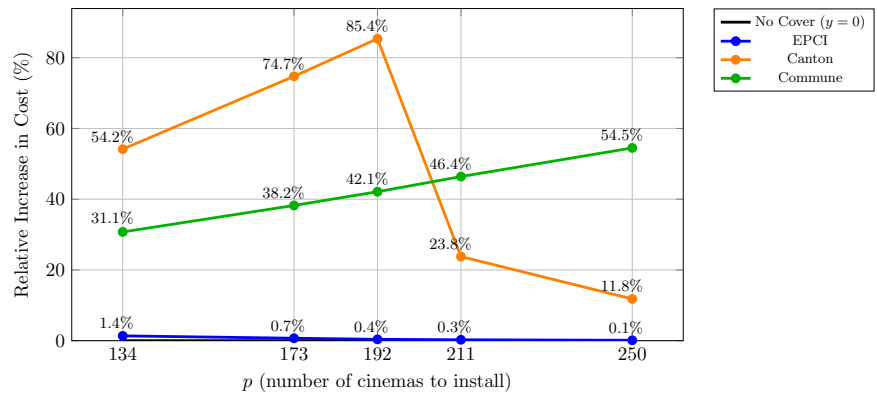

{Table~\ref{tab:distribution-travel-time-p-192} compares travel time distributions for $p = 192$, corresponding to the actual number of cinemas in the PACA region. The second column reports real cinema locations, mapped for \texttt{paca2641} and assigned using the relaxed GAP formulation. The remaining columns show solutions for the classical C$p$MP$^r$ and those with \emph{EPCI}, \emph{canton}, and \emph{commune} coverage, as well as a capacitated Maximum Coverage Location Problem \citep{Church1974} solution with a 20-minute cutoff (MCLP-20min), including percentages of demand within travel time intervals, summary statistics (average and maximum travel time), the Gini coefficient of the travel time distribution, and the number of spatial units covered.}

\FA{Real cinema locations are less efficiency oriented than the classical C$p$MP$^r$ solution, with more people over 30 minutes away and higher average and maximum travel times. This also reflects cinemas installed over time, as population distribution has since changed. Both solutions cover a similar number of territorial units, indicating that efficiency-based territorial logic aligns with the actual distribution. Compared to C$p$MP$^r$, adding coverage constraints has different effects depending on the type: \emph{EPCI} coverage minimally impacts costs and travel times. For \emph{Communes}, coverage increases the number of people over 20 minutes away, but reduces the maximum travel time and the Gini coefficient relative, indicating a more equitable distribution despite the longer tail. \emph{Cantons} coverage results in the largest $\mathrm{RIC}$ increase among the C$p$MP$^r$ variants, reducing the population within 5 minutes by over 20\%, but decreasing maximum travel distance and the Gini coefficient among them, directly reflecting its equity objective. The MCLP-20min solution enforces its cutoff exactly, with no demand assigned beyond 20 minutes and the lowest overall Gini coefficient (0.177), showing that a more equitable distribution does not necessarily imply shorter average travel distances. Here, a clear trade-off appears within the interval structure: the percentage served within 5 minutes drops sharply, in exchange for a much higher percentage served within 20 minutes. However, this comes with an efficiency cost, with an increase of $123.49\%$ in $\mathrm{RIC}$. In general, tcoverage constraints improve accessibility while preserving efficient-location logic, and such constraints are sometimes obligations of some service installation process.}



    

\begin{table}[htbp]
\centering
\caption{Comparison of cinema distributions for solutions with and without territorial coverages for $p=192$ for \texttt{paca2641}.}
\label{tab:distribution-travel-time-p-192}
{\tiny

\begin{tabular}{l@{\hspace{1cm}}r@{\hspace{1cm}}r@{\hspace{1cm}}r@{\hspace{1cm}}r@{\hspace{1cm}}r@{\hspace{1cm}}r}
\toprule
\multirow{2}{*}{\textbf{Analysis}} & \multicolumn{6}{c}{\textbf{Solutions}} \\
\cmidrule(lr){2-7}
  & \textbf{Real} & \textbf{C$p$MP$^r$} & \textbf{Cover EPCIs} & \textbf{Cover Canton} & \textbf{Cover Commune} & \textbf{MCLP-20min} \\
\midrule
\textbf{Interval (minutes)}        & \multicolumn{6}{c}{\textbf{Percentage of demand (\%)}} \\
\cmidrule(lr){1-7}
0--5    & 33.94 & 46.53 & 46.01 & 25.50 & 34.18 & 4.22 \\
5--10    & 11.22 & 19.12 & 19.57 & 18.31 & 20.90 & 10.34 \\
10--20    & 29.04 & 30.97 & 31.01 & 37.22 & 32.11 & 85.44 \\
20--30    & 12.13 & 2.83 & 2.86 & 12.14 & 11.95 & 0.00 \\
30+    & 13.67 & 0.55 & 0.55 & 6.83 & 0.86 & 0.00 \\
\cmidrule(lr){1-7}
\textbf{Statistics}        & \multicolumn{6}{c}{\textbf{Time (minutes)}} \\
\cmidrule(lr){1-7}
Avg. distance     & 13.93 & 6.43 & 6.46 & 11.93 & 9.15 & 14.43 \\
Max distance    & 77.89 & 73.83 & 61.97 & 57.73 & 61.58 & 20.00 \\
Gini coefficient    & 0.563 & 0.568 & 0.563 & 0.456 & 0.496 & 0.177 \\
\cmidrule(lr){1-7}
\textbf{Territorial division}        & \multicolumn{6}{c}{\textbf{Number of spatial units covered}} \\
\cmidrule(lr){1-7}
EPCIs     & 45 & 46 & 51 & 51 & 48 & 49 \\
Canton     & 113 & 116 & 119 & 192 & 133 & 127 \\
Commune     & 145 & 145 & 146 & 192 & 192 & 174 \\
\bottomrule
\end{tabular}
}
\end{table}

\section{Conclusion}
\label{sec:Conclusion}

This paper addressed the Capacitated $p$-Location Problem (C$p$LP) and its territorial coverage variants, providing strengthened ILP formulations and the Random Sampling Spatial Voting (RSSV) matheuristic. Our computational study demonstrates RSSV effectiveness in producing high-quality solutions within one hour on both literature benchmarks and real-world instances, requiring minimal parameter tuning. We conclude that our methods prove effective across diverse datasets, showing RSSV versatility for location problems. Furthermore, territorial constraints substantially impact facility layouts, revealing that for a given number of facilities, ensuring coverage of subareas can be achieved but with a loss in cost efficiency. However, improvements can be made to spatial equity.



Overall, our contributions provide a methodology for incorporating equitable territorial considerations into facility location problems, combining ILP modeling, a competitive matheuristic, and real-case study analysis. Thus a robust method for practical application has been developed. Future work includes \textcolor{black}{the study of service relocation, addition, and removal, through a more detailed analysis of other types of services and the varying costs of leaving an area without service coverage, as well as} incorporating temporal dimensions, accounting for cases where a subset of facilities may be mobile during defined time periods.

\bibliographystyle{plainnat}
\bibliography{sample}

\newpage
\appendix
\section{Supplementary Material / Appendix}
\label{sec:Appendix}


\subsection{Additional Computational Experiments}
\label{sec:appendix-experiments}

\subsubsection{Literature \textit{vs.} RSSV(C$p$MP)}
\label{sec:Exp_Literature-vs-RSSV}

We evaluated the competitiveness of the RSSV method for solving the C$p$MP using the three literature instance sets described in Section~\ref{sec:CompExp-Instances}, which are similar in size to, or smaller than, our real case study in the PACA region in terms of the sizes of $I$ and $J$. As previously noted, these instances follow a particular structure: all candidate locations have equal capacities for each value of $p$, and the distance metric is Euclidean. We tested several values of the RSSV sampling parameter $n_{\text{cand}}$ to identify one that performs consistently well across different values of $p$.

We compared RSSV against three state-of-the-art approaches: the IRMA matheuristic from~\cite{stefanello2015matheuristics}, the scalable matheuristic from~\cite{gnagi2021matheuristic}, \textcolor{black}{and the large neighborhood search from~\cite{gjergji2026large}.}
The Gap [\%] is computed relative to the best-known solution reported by~\cite{stefanello2015matheuristics}, while method comparisons use the best results reported in the more recent work of~\cite{gnagi2021matheuristic} and~\cite{gjergji2026large}.

Tables~\ref{tab:rssv_cpmp_lns_p3038_fnl4461} and~\ref{tab:rssv_cpmp_SJC} present the results, with the first three columns defined as in Table~\ref{tab:gap_comparison_arcgis}. The ratio $|J|/p$ measures facility density for each instance. Column ``\textbf{GB21}\textsubscript{best}'' reports the Gap [\%] and CPU time (Time [s]) of the best solution obtained between the two implementations in~\cite{gnagi2021matheuristic}: their matheuristic and their reimplementation of IRMA~\cite{stefanello2015matheuristics}. \textcolor{black}{Column ``\textbf{IG26}\textsubscript{LNS}'' reports the corresponding results for the large neighborhood search of~\cite{gjergji2026large}}. The remaining columns show solutions obtained with the ``ILP(C$p$MP)'' formulation from Section~\ref{sec:ProbDef} and with the ``RSSV(C$p$MP, $n_{\text{cand}}$)'' matheuristic, for two values of $n_{\text{cand}}$: $|J|$, where all locations are included, and $2p + 0.1|J|$, a reduced set used both during sampling and in the final problem phase.

The value ``\texttt{TL}'' (Time Limit) indicates that the solver reached the imposed CPU time limit of 3600 seconds; ``\texttt{ML}'' (Memory Limit) indicates that 128GB of RAM was insufficient to complete the CPLEX ILP solver run.

From the results in Table~\ref{tab:rssv_cpmp_lns_p3038_fnl4461}, we conclude that the performance of the RSSV method depends on the number of facilities $p$ and the ratio $|J|/p$. When $p$ is large (i.e., $|J|/p$ is small), RSSV($|J|$) leads to better solutions. In contrast, for small values of $p$ (i.e., large $|J|/p$ ratios), the reduced location set RSSV($2p + 0.1|N|$) performs better. This indicates that adjusting the number of candidate locations based on the instance size and density helps find a good solution quality for the instances just by changing the sampling and final problem candidates sizes. 

Table~\ref{tab:rssv_cpmp_SJC} presents the results obtained on the benchmark instance set SJC. The results show that our approach does not reduce the optimality gap as quickly as the methods from the literature, consistent with our expectations that it is not the most effective option for small instances, whereas the ILP method, run without a time limit, finds the optimal solution in less than 600 seconds for all of these instances.


\begin{table}[htbp]
\centering
\caption{Comparison of RSSV(C$p$MP) with the best solutions for the instances p3038 and fnl4461.}
\label{tab:rssv_cpmp_lns_p3038_fnl4461}
{\tiny

\begin{tabular}{
        l 
        S[table-format=4.0] 
        S[table-format=4.0] 
        S[table-format=3.1] 
        S[table-format=3.2, detect-weight] 
        S[table-format=4.0] 
        S[table-format=2.2, detect-weight] 
        S[table-format=4.0] 
        S[table-format=2.2, detect-weight] 
        S[table-format=4.0] 
        S[table-format=2.2, detect-weight] 
        S[table-format=4.0] 
        S[table-format=2.2, detect-weight] 
        S[table-format=4.0] 
    }
    \toprule
    \textbf{Instance} & \(\boldsymbol{|J| = |I|}\) & \(\boldsymbol{p}\) & \(\boldsymbol{|J|/p}\) &
    \multicolumn{2}{c}{\textbf{ILP}} &
    \multicolumn{2}{c}{\textbf{GB21}\textsubscript{best}} &
    \multicolumn{2}{c}{\textbf{IG26}\textsubscript{LNS}} &
    \multicolumn{2}{c}{\textbf{RSSV}(\(\boldsymbol{|J|}\))} &
    \multicolumn{2}{c}{\textbf{RSSV}(\(\boldsymbol{2p + 0.1|J|}\))} \\
    \cmidrule(lr){5-6} \cmidrule(lr){7-8} \cmidrule(lr){9-10} \cmidrule(lr){11-12} \cmidrule(lr){13-14}
    & & & 
    & {Gap [\%]} & {CPU [s]} 
    & {Gap [\%]} & {CPU [s]} 
    & {Gap [\%]} & {CPU [s]} 
    & {Gap [\%]} & {CPU [s]} 
    & {Gap [\%]} & {CPU [s]} \\
    \hline
    p3038\_600      & 3038 & 600  & 5.1  & 142.32 & \texttt{TL} & 0.03 & \texttt{TL} & \bfseries 0.01 & \texttt{TL} & 0.09 & \texttt{TL} & 6.60 & \texttt{TL} \\
    p3038\_700      & 3038 & 700  & 4.3  & 26.33  & \texttt{TL} & 0.03 & \texttt{TL} & \bfseries 0.00 & \texttt{TL} & 0.02 & \texttt{TL} & 5.50 & \texttt{TL} \\
    p3038\_800      & 3038 & 800  & 3.8  & 8.51   & \texttt{TL} & 0.03 & \texttt{TL} & 0.02 & \texttt{TL} & \bfseries 0.01 & \texttt{TL} & 5.30 & \texttt{TL} \\
    p3038\_900      & 3038 & 900  & 3.4  & 67.28  & \texttt{TL} & 0.02 & \texttt{TL} & \bfseries 0.01 & \texttt{TL} & \bfseries 0.01 & \texttt{TL} & 3.70 & \texttt{TL} \\
    p3038\_1000     & 3038 & 1000 & 3.0  & \texttt{ML} & \texttt{TL} & 0.04 & \texttt{TL} & \bfseries 0.00 & \texttt{TL} & \bfseries 0.00 & \texttt{TL} & 3.00 & \texttt{TL} \\
    fnl4461\_0020   & 4461 & 20   & 223.1 & \texttt{ML} & \texttt{TL} & 0.28 & \texttt{TL} & \bfseries 0.12 & \texttt{TL} & 42.86 & \texttt{TL} & 0.13 & \texttt{TL} \\
    fnl4461\_0100   & 4461 & 100  & 44.6  & \texttt{ML} & \texttt{TL} & 0.55 & \texttt{TL} & 0.35 & \texttt{TL} & 28.95 & \texttt{TL} & \bfseries 0.07 & \texttt{TL} \\
    fnl4461\_0250   & 4461 & 250  & 17.8  & \texttt{ML} & \texttt{TL} & 0.46 & \texttt{TL} & \bfseries 0.14 & \texttt{TL} & 28.73 & \texttt{TL} & 0.25 & \texttt{TL} \\
    fnl4461\_0500   & 4461 & 500  & 8.9   & \texttt{ML} & \texttt{TL} & 0.23 & \texttt{TL} & \bfseries 0.03 & \texttt{TL} & 0.13 & \texttt{TL} & 1.06 & \texttt{TL} \\
    fnl4461\_1000   & 4461 & 1000 & 4.5   & \texttt{ML} & \texttt{TL} & 0.03 & \texttt{TL} & \bfseries 0.01 & \texttt{TL} & \bfseries 0.01 & \texttt{TL} & 1.68 & \texttt{TL} \\ 
    \hline
    \multicolumn{14}{l}{\texttt{TL}: Time limit of 3600 seconds} \\
    \multicolumn{14}{l}{\texttt{ML}: Memory limit of 128 GB RAM (out of RAM)}  \\
    \bottomrule
\end{tabular}
}
\end{table}

\begin{table}[htbp]
\centering
\caption{Comparison of RSSV(C$p$MP) with the optimal solutions for the instances SJC.}
\label{tab:rssv_cpmp_SJC}
{\tiny
\begin{tabular}{
    l 
    S[table-format=4.0] 
    S[table-format=4.0] 
    S[table-format=3.1] 
    S[table-format=2.2, detect-weight, tight-spacing=true] 
    S[table-format=2.2] 
    S[table-format=2.2, detect-weight, tight-spacing=true] 
    S[table-format=2.2] 
    S[table-format=2.2, detect-weight, tight-spacing=true] 
    S[table-format=2.2] 
    S[table-format=2.2, detect-weight, tight-spacing=true] 
    S[table-format=2.2] 
}
    \toprule
        \textbf{Instance} & $\boldsymbol{|J| = |I|}$ & $\boldsymbol{p}$ & $\boldsymbol{|J|/p}$ &
        \multicolumn{2}{c}{\textbf{ILP}} &
        \multicolumn{2}{c}{\textbf{GB21}\textsubscript{best}} &
        \multicolumn{2}{c}{\textbf{RSSV($\boldsymbol{|J|}$)}} &
        \multicolumn{2}{c}{\textbf{RSSV($\boldsymbol{\frac{|J|}{2}}$)}} \\
        \cmidrule(lr){5-6} \cmidrule(lr){7-8} \cmidrule(lr){9-10} \cmidrule(lr){11-12}
        & & & & {Gap [\%]} & {CPU [s]} & {Gap [\%]} & {CPU [s]} & {Gap [\%]} & {CPU [s]} & {Gap [\%]} & {CPU [s]} \\ 
    \hline
        SJC1      & 100 & 10  & 10.0  & \bfseries 0.00 & 10.78 & \bfseries 0.00 & 3.94   & 0.59         & 13.63 & \bfseries 0.00         & 3.93 \\
        SJC2      & 200 & 15  & 13.3  & 0.07 & \texttt{TL} & \bfseries 0.00 & 4.78   & \bfseries 0.00 & 21.98 & \bfseries 0.00         & 16.01 \\
        SJC3a     & 300 & 25  & 12.0  & 2.67 & \texttt{TL} & \bfseries 0.00 & 28.21  & 0.33         & \texttt{TL} & 0.20         & 19.22 \\
        SJC3b     & 300 & 30  & 10.0  & \bfseries 0.00 & \texttt{TL} & \bfseries 0.00 & 14.84  & \bfseries 0.00 & 20.82 & 0.70        & 7.63 \\
        SJC4a     & 402 & 30  & 13.4  & 44.70 & \texttt{TL} & \bfseries 0.00 & 42.71  & 13.09  & \texttt{TL} & 1.00         & \texttt{TL} \\
        SJC4b     & 402 & 40  & 10.1  & 35.90 & \texttt{TL} & \bfseries 0.00 & 24.18  & 0.06        & \texttt{TL} & 0.50         & 12.11 \\
    \hline
        \multicolumn{12}{l}{\texttt{TL}: Time limit of 40 seconds} \\
    \bottomrule
\end{tabular}

}
\end{table}


\subsubsection{Network Analyst \textit{vs.} ILP(C$p$MP$^r$) \textit{vs.} RSSV(C$p$MP$^r$)}
\label{sec:Exp_ArcGIS-vs-MIP-vs-RSSV}


We benchmarked RSSV against a commercial solver for the C$p$MP$^r$, using the previously described instance \texttt{paca2641}. Network Analyst is an extension of ArcGIS Pro and remains widely used by geographers and planners because of its user-friendly interface and large user community. This tool provides functionalities for solving network based optimization problems, such as routing, service area delineation, location-allocation, and closest facility analysis. It uses a network dataset to model real-world transportation systems and supports various problem types, including the location-allocation analysis layer. Within this layer, the \emph{Minimize Weighted Impedance} problem most closely matches C$p$MP$^r$ : it minimizes total travel time between demand points and selected facilities, using a multi-stage heuristic based on the vertex substitution method of \cite{teitz1968heuristic}. However, this method does not account for capacity constraints and instead solves the standard $p$-median problem ($p$MP). To enable a meaningful comparison with the capacitated version, we used the $p$ facilities locations selected by the ArcGIS Pro method and performed an optimal assignment of demand respecting facility capacities. The assignment problem is formulated as a Generalized Assignment Problem (GAP), where demand can be split among multiple facilities. If the total assigned demand exceeds the available capacity, the solution becomes infeasible under the constraints.  This problem is solved optimally within a few seconds, and therefore, its computation time was not considered in the final evaluation of the ArcGIS Pro method. We compared the solution obtained with those produced by methods that can be directly applied to the C$p$MP$^r$.

{
Table~\ref{tab:gap_comparison_arcgis} compares the solution quality, expressed as the Gap [\%], and execution time (Time [s]) for the \texttt{paca2641} dataset obtained with three approaches: ArcGIS Pro, ILP (C$p$MP$^r$), and RSSV (C$p$MP$^r$, $|J|$).
The column `Instance'' indicates the instance name, where ``\(\lvert J \rvert = \lvert I \rvert\)'' represents the number of customers and potential facility locations, and $p$ is the number of facilities to be installed. ArcGIS Pro required around 840 seconds to process each instance, therefore, this value was adopted as a fixed total time limit (``\texttt{TL}'') to others methods to ensure a fair comparison of computational performance among all methods.
A dash (``--'') denotes that ArcGIS Pro failed to generate a feasible solution due to capacity constraint violations. Among the three methods, RSSV($|J|$) produced the best solutions within the time limit of 840 seconds. 

These results were expected, since ArcGIS Pro implements a generic heuristic method for $p$MP. However, they highlight the need for efficient and practical methods to solve location problems, especially when the inclusion of territorial coverage constraints is necessary, as their implementation is not available in the current version of ArcGIS Pro Network Analyst.

}

\begin{table}[htbp]
\centering
\caption{GAP comparison for ArcGIS, ILP, and RSSV approach on \texttt{paca2641} for C$p$MP$^r$.}
\label{tab:gap_comparison_arcgis}
{\tiny
\begin{tabular}{
        l  
        S[table-format=4.0]  
        S[table-format=4.0]  
        S[table-format=2.2, detect-weight]  
        S[table-format=4.0]  
        S[table-format=2.2, detect-weight]  
        S[table-format=4.0]  
        S[table-format=2.2, detect-weight]  
        S[table-format=4.0]  
    }
    \toprule
        \textbf{Instance} & \(\boldsymbol{|J| = |I|}\) & \(\boldsymbol{p}\) &
        \multicolumn{2}{c}{\textbf{ArcGIS Pro}} & 
        \multicolumn{2}{c}{\textbf{ILP}} & 
        \multicolumn{2}{c}{\textbf{RSSV($|J|$)}} \\
        \cmidrule(lr){4-5} \cmidrule(lr){6-7} \cmidrule(lr){8-9} 
        & & & {Gap [\%]} & {Time [s]} & {Gap [\%]} & {Time [s]} & {Gap [\%]} & {Time [s]}  \\
    \hline
    \texttt{paca2641}        & 2641 & 134 & {--}    & {--}    & 80.31 & \texttt{TL} & \bfseries 15.37 & \texttt{TL} \\
    \texttt{paca2641}        & 2641 & 173 & 179.07   & 822     & 142.75 & \texttt{TL} & \bfseries 6.78  & \texttt{TL} \\
    \texttt{paca2641}        & 2641 & 192 & 133.58   & 864     & 79.19 & \texttt{TL} & \bfseries 7.57  & \texttt{TL} \\
    \texttt{paca2641}        & 2641 & 211 & 129.13   & 893     & 17.96 & \texttt{TL} & \bfseries 8.94  & \texttt{TL} \\
    \texttt{paca2641}        & 2641 & 250 & 91.55   & 879     & 20.70 & \texttt{TL} & \bfseries 1.21  & \texttt{TL} \\
    \midrule
    \multicolumn{9}{l}{\texttt{TL}: Time limit of 840 seconds} \\
    \bottomrule
\end{tabular}
}
\end{table}

\subsection{NP-Hardness of C$p$LP-TC}
\label{app:hardness_cpmptc}

{\color{black}
\begin{lemma}\label{lem:cpmp_tc_hard_p_gt_m}
C$p$LP-TC is NP-hard for instances with $p>m^s$.
\end{lemma}
\begin{proof}
Take an instance $\langle J, I, p, W, R \rangle$ of C$p$LP \citep{Kariv1979} with $p\ge2$.
Let $d=m^s-1$. Build a C$p$LP-TC instance by adding $d\ge1$ dummy pairs $(i_k,j_k)$ in sets $I$ and $J$, $k=1,\dots,d$,
with $\mathrm{dist}(i_k,j_k)=0$, $R_{j_k}=W_{i_k}$, and
$\mathrm{dist}(i,j)=\infty$ whenever exactly one of $i,j$ is a dummy element.
Define $S_1=I$, $J(S_1)=J$, and $S_{k+1}=\{i_k\}$, $J(S_{k+1})=\{j_k\}$, so
$m^s=d+1$. Set $p'=p+d$. Since $p\ge2$, $p'=p+d>1+d=m^s$, so constraint
(\ref{cover_all_subar}) applies. See Figure~\ref{fig:lemma1_construction}

Because $J(S_{k+1})=\{j_k\}$ is a singleton, (\ref{cover_all_subar}) forces
$y_{j_k}=1$ for every $k$. The distances defined for the dummy locations force $i_k$ to be
served by $j_k$ at zero cost and prevent any real customer from being served by a dummy facility. This uses exactly $d$ of the $p'$ facilities,
the remaining $p$ must be opened in $J$ to serve $I$, giving a feasible
C$p$LP solution of the same cost, and conversely. Hence, the reduction is polynomial and preserves the optimal value. Therefore, C$p$LP-TC is NP-hard for $p > m^s$.
\end{proof}

\begin{lemma}\label{lem:cpmp_tc_hard_p_leq_m}
C$p$LP-TC is NP-hard for instances with $p\le m^s$.
\end{lemma}
\begin{proof}
Take any instance $\langle J, I, p, W, R \rangle$ of C$p$LP \citep{Kariv1979}. Build a C$p$LP-TC
instance with the same $\langle J, I, p, W, R \rangle$, setting $m^s=|J|$ and
$S=\{\{j\}:j\in J\}$. Since $p\le|J|=m^s$ and each $J(S_k)$ is a singleton, any solution obtained for  C$p$LP-TC is a solution for the C$p$LP. We conclude
C$p$LP-TC is NP-hard for $p\le m^s$.
\end{proof}
}

\subsection{Additional figures}
\label{sec:appendix-figures}

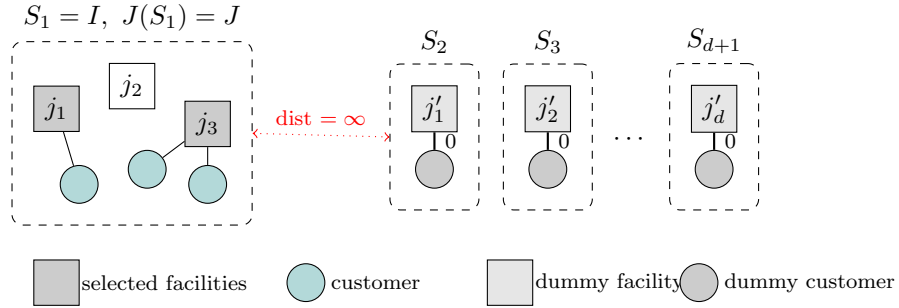
\begin{figure}[htbp]
\centering
\begin{tikzpicture}[
    facnotopen/.style={rectangle, draw, minimum size=6mm},
    facility/.style={rectangle, draw, fill=black!20, minimum size=6mm},
    customer/.style={circle, draw, fill=teal!30, minimum size=5mm},
    dummy/.style={rectangle, draw, fill=gray!20, minimum size=6mm},
    dummycust/.style={circle, draw, fill=gray!40, minimum size=5mm},
    subarea/.style={draw, dashed, rounded corners, inner sep=8pt},
    every node/.append style={font=\small}
]

\begin{scope}
\node[facility] (j1) at (0,1) {$j_1$};
\node[facnotopen] (j2) at (1,1.3) {$j_2$};
\node[facility] (j3) at (2,0.8) {$j_3$};
\node[customer] (i1) at (0.3,0) {};
\node[customer] (i2) at (1.2,0.2) {};
\node[customer] (i3) at (2,0) {};

\draw[thin] (i1) -- (j1);
\draw[thin] (i2) -- (j3);
\draw[thin] (i3) -- (j3);

\node[subarea, fit=(j1)(j2)(j3)(i1)(i2)(i3), label={[font=\small]above:$S_1 = I,\ J(S_1)=J$}] (S1) {};
\end{scope}

\node[dummy] (jd1) at (5,1) {$j_1'$};
\node[customer, fill=gray!40] (id1) at (5,0.2) {};
\node[subarea, fit=(jd1)(id1), label={[font=\small]above:$S_2$}] (Sd1) {};
\draw[thick] (jd1) -- (id1) node[midway, right] {\scriptsize $0$};

\node[dummy] (jd2) at (6.5,1) {$j_2'$};
\node[customer, fill=gray!40] (id2) at (6.5,0.2) {};
\node[subarea, fit=(jd2)(id2), label={[font=\small]above:$S_3$}] (Sd2) {};
\draw[thick] (jd2) -- (id2) node[midway, right] {\scriptsize $0$};

\node at (7.6,0.6) {$\cdots$};

\node[dummy] (jdd) at (8.7,1) {$j_d'$};
\node[customer, fill=gray!40] (idd) at (8.7,0.2) {};
\node[subarea, fit=(jdd)(idd), label={[font=\small]above:$S_{d+1}$}] (Sdd) {};
\draw[thick] (jdd) -- (idd) node[midway, right] {\scriptsize $0$};

\draw[<->, dotted, red] (S1.east) -- (Sd1.west) node[midway, above, font=\scriptsize] {dist $=\infty$};

\node[facility] at (0,-1.3) {};
\node[right=2mm of {(0,-1.3)}, font=\scriptsize] {selected facilities};
\node[customer] at (3.3,-1.3) {};
\node[right=2mm of {(3.3,-1.3)}, font=\scriptsize] {customer};
\node[dummy] at (6,-1.3) {};
\node[right=2mm of {(6,-1.3)}, font=\scriptsize] {dummy facility};
\node[dummycust] at (8.5,-1.3) {};
\node[right=2mm of {(8.5,-1.3)}, font=\scriptsize] {dummy customer};

\end{tikzpicture}
\caption{Construction for Lemma~\ref{lem:cpmp_tc_hard_p_gt_m}: the original C$p$LP instance is placed in a single subarea $S_1$, while $d$ dummy subareas $S_2,\dots,S_{d+1}$ each contain one isolated dummy pair $(i_k',j_k')$ at distance $0$.
}
\label{fig:lemma1_construction}
\end{figure}

\begin{figure}[htbp]
    \caption{Flowchart of the RSSV matheuristic. The procedure consists of five main steps. When time limits allow, a sixth post-optimization step can further refine the final result.}
    \hspace{-2cm} 
    \centering
    \begin{tikzpicture}[scale=0.6, node distance=2cm, every node/.style={transform shape, font=\Large}, every brace/.append style={line width=0.7pt}]
\node (start) [startstop] {Original Dataset $\langle J, I, p, W, R \rangle$};
\node (sample1) [process, below left of=start, xshift=-4cm, yshift=-1cm] {\( \langle J_1, I, p, W, R \rangle \)};
\node (sample2) [process, below of=start, yshift=-0.5cm] {...};
\node (sampleM) [process, below right of=start, xshift=4cm, yshift=-1cm] {\( \langle J_M, I, p, W, R \rangle \)};

\node (group1) [groupbox, fit=(sample1)(sampleM)(sample2)] {};
\node at (group1.west) [left, xshift=-0.5cm, align=center] {\textbf{R}andom \\ \textbf{S}ampling};

\node (sol1) [process, below of=sample1, yshift=-1cm] {$f_1$};
\node (sol2) [process, below of=sample2, yshift=-1cm] {...};
\node (solM) [process, below of=sampleM, yshift=-1cm] {$f_M$};

\node (group2) [groupbox, fit=(sol1)(solM)(sol2)] {};
\node at (group2.west) [left, xshift=-0.5cm, align=center] {Sub-Problem \\ Solving};

\node (voting) [process, below of=sol2, yshift=-0.7cm] {Voting};

\node (group3) [groupbox, fit=(voting), inner xsep=5pt, inner ysep=5pt] {};
\node at (group3.west) [left, xshift=-4.0cm, align=center] {\textbf{S}patial \\ \textbf{V}oting};

\node (finalSubproblem) [process, below of=voting, yshift=-0.5cm] {\( \langle J_{n_{\text{cand}}}, I, p, W, R \rangle \)};
\node (finalSolution) [process, below of=finalSubproblem, yshift=-0.5cm] {$f^*$};

\node (postOptimization) [process, right of=finalSolution, xshift=5cm] {Post- Optimization};

\node (filtering) [groupbox, fit=(finalSubproblem), inner xsep=3pt, inner ysep=3pt] {}; 
\node at (filtering.west) [left, xshift=-3.5cm, align=center] {Filtering};

\node (finalsol) [groupbox, fit=(finalSolution), inner xsep=3pt, inner ysep=3pt] {}; 
\node at (finalsol.west) [left, xshift=-3cm, align=center] {Final Problem \\ Solving};

\draw [arrow] (start) -- (sample1);
\draw [arrow] (start) -- (sample2);
\draw [arrow] (start) -- (sampleM);

\draw [arrow] (sample1) -- (sol1);
\draw [arrow] (sample2) -- (sol2);
\draw [arrow] (sampleM) -- (solM);

\draw [arrow] (sol1) -- (voting);
\draw [arrow] (sol2) -- (voting);
\draw [arrow] (solM) -- (voting);

\draw [arrow] (voting) -- (finalSubproblem);
\draw [arrow] (finalSubproblem) -- (finalSolution);

\draw [arrow] (finalSolution.east) -- (postOptimization.west);




\end{tikzpicture}
    
    \label{fig:rssv_flowchart}
\end{figure}
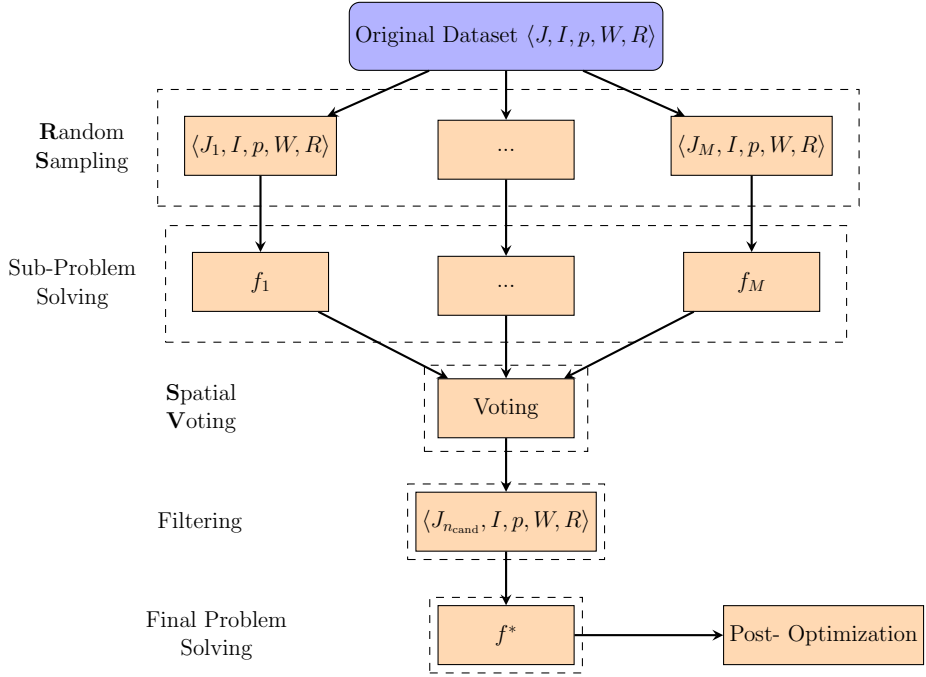

\begin{algorithm}[htbp]
\caption{Post-Optimization via Progressive Neighborhood Expansion}
\label{alg:post_optimization}
\vskip4pt
\begin{algorithmic}
\Procedure{PostOptimize}{$f^*, J, I, p, W, R, t_{\text{rem}}$}
    \Comment{$t_{\text{rem}}$: Remaining time available}
    \State $f^+ \gets f^*$
    \State $v \gets 1$
    \While{$t_{\text{rem}} > 0$ \textbf{and} $|J'| < |J|$}
        \State $J' \gets f^+$
        \For{each $f_g \in f^+$}
            \State $j' \gets v$-th closest neighbor of $f_g$ in $J \setminus J'$
            \State $J' \gets J' \cup \{j'\}$
        \EndFor
        \State $D \gets \max_{i \in I} d(i, f^+(i))$ \Comment{Max distance from customers to assigned facility}
        \State $\tilde{f} \gets$ \Call{SolveILP}{$J', I, p, W, R, D$}
        \If{$\tilde{f}$ improves $f^+$}
            \State $f^+ \gets \tilde{f}$
            \State $v \gets 1$
        \Else
            \State $v \gets v + 1$
        \EndIf
        \State update $t_{\text{rem}}$
    \EndWhile
    \State \textbf{return} $f^+$
\EndProcedure
\Statex
\end{algorithmic}
\end{algorithm}

\begin{figure}[htbp]
    \caption{Population distribution of \texttt{paca2641} dataset. Areas shown in red correspond to higher population density, while lighter tones indicate sparsely populated zones, highlighting the strong spatial heterogeneity of the region in terms of population size.}
    \centering
    \includegraphics[width=0.6\linewidth]{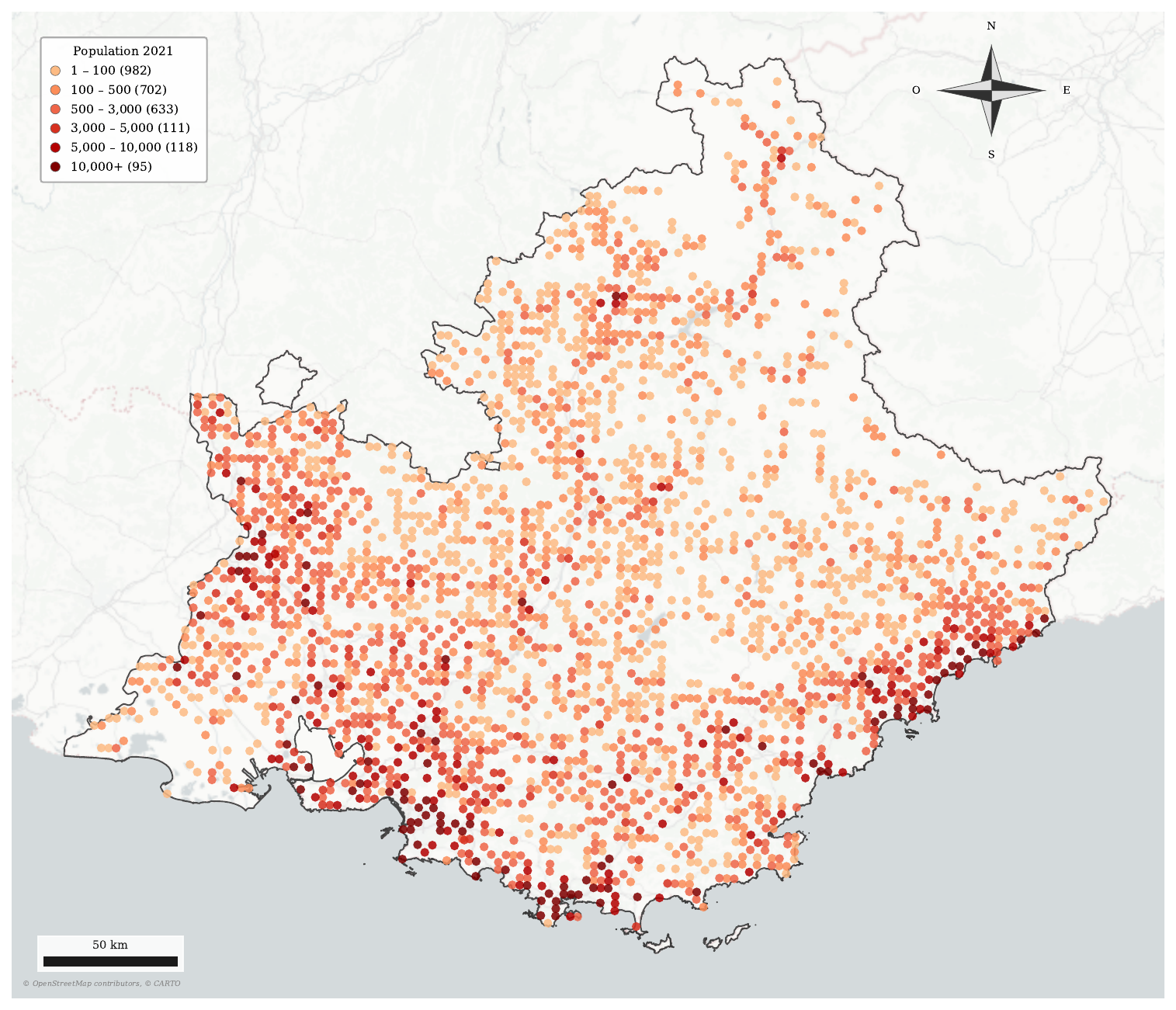}
    
    \label{fig:paca2641_population_distribution}
\end{figure}

\begin{figure}[htbp]
    \caption{Spatial analysis for fixed $p=250$ and $n_{\text{cand}}=1600$ with different $\kappa$ values for the \texttt{paca5282} instance.}
    \centering
    \begin{subfigure}[t]{0.32\textwidth}
        \centering
        \includegraphics[width=\linewidth]{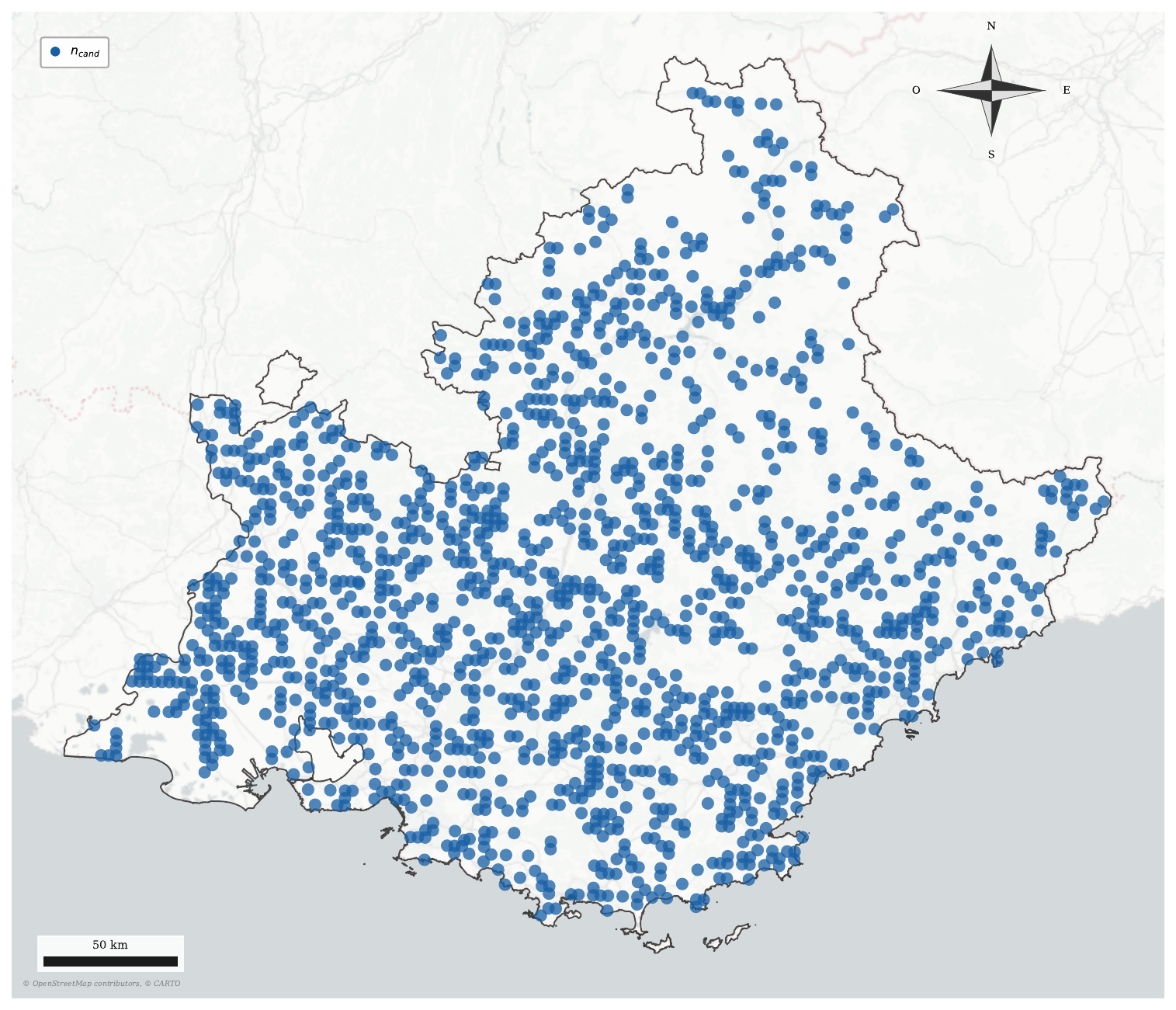}
        \caption{$\kappa = 1$}
        \label{fig:spatial_p5282_k1}
    \end{subfigure}
    \hfill
    \begin{subfigure}[t]{0.32\textwidth}
        \centering
        \includegraphics[width=\linewidth]{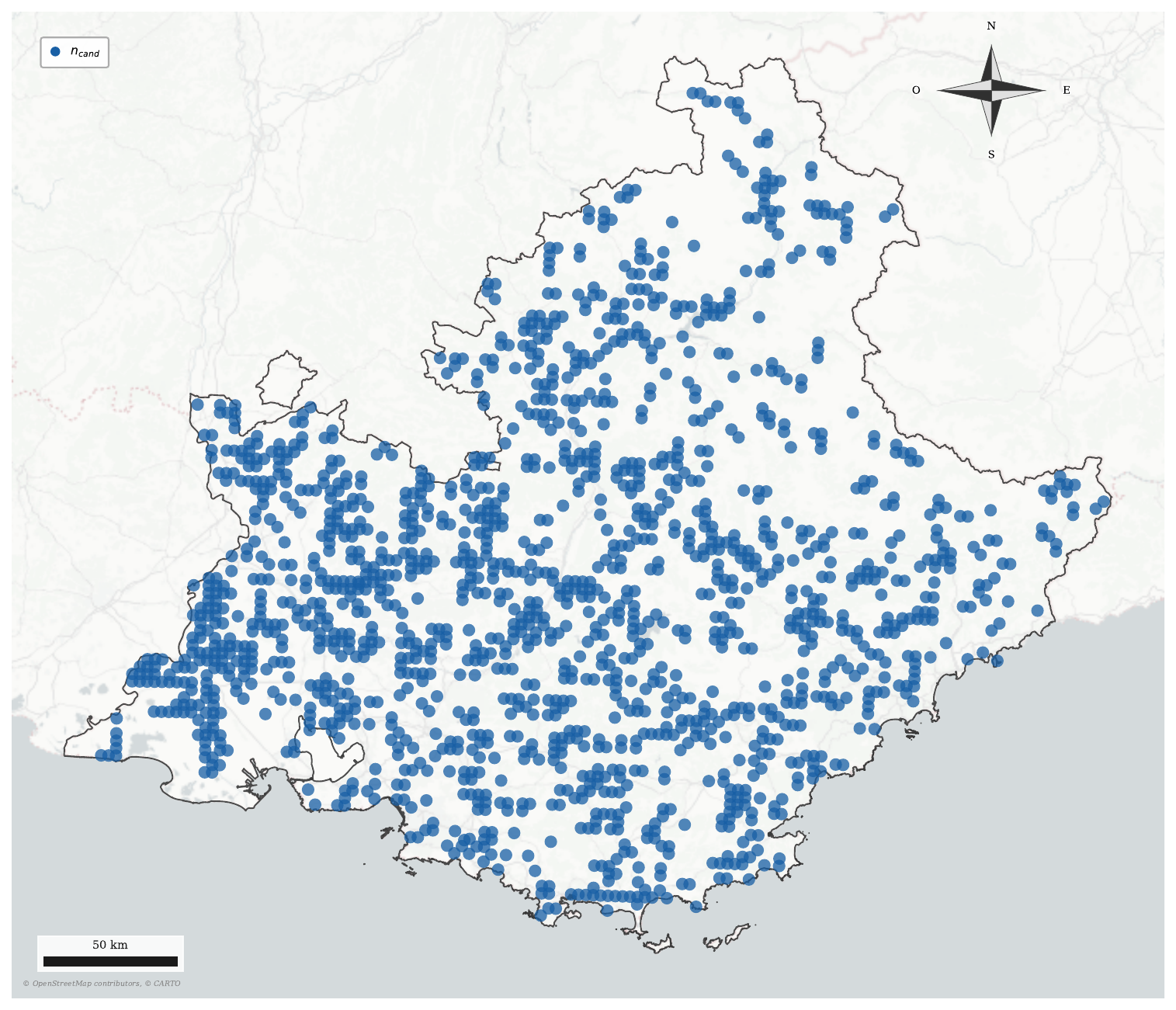}
        \caption{$\kappa = 2$}
        \label{fig:spatial_p5282_k2}
    \end{subfigure}
    \hfill
    \begin{subfigure}[t]{0.32\textwidth}
        \centering
        \includegraphics[width=\linewidth]{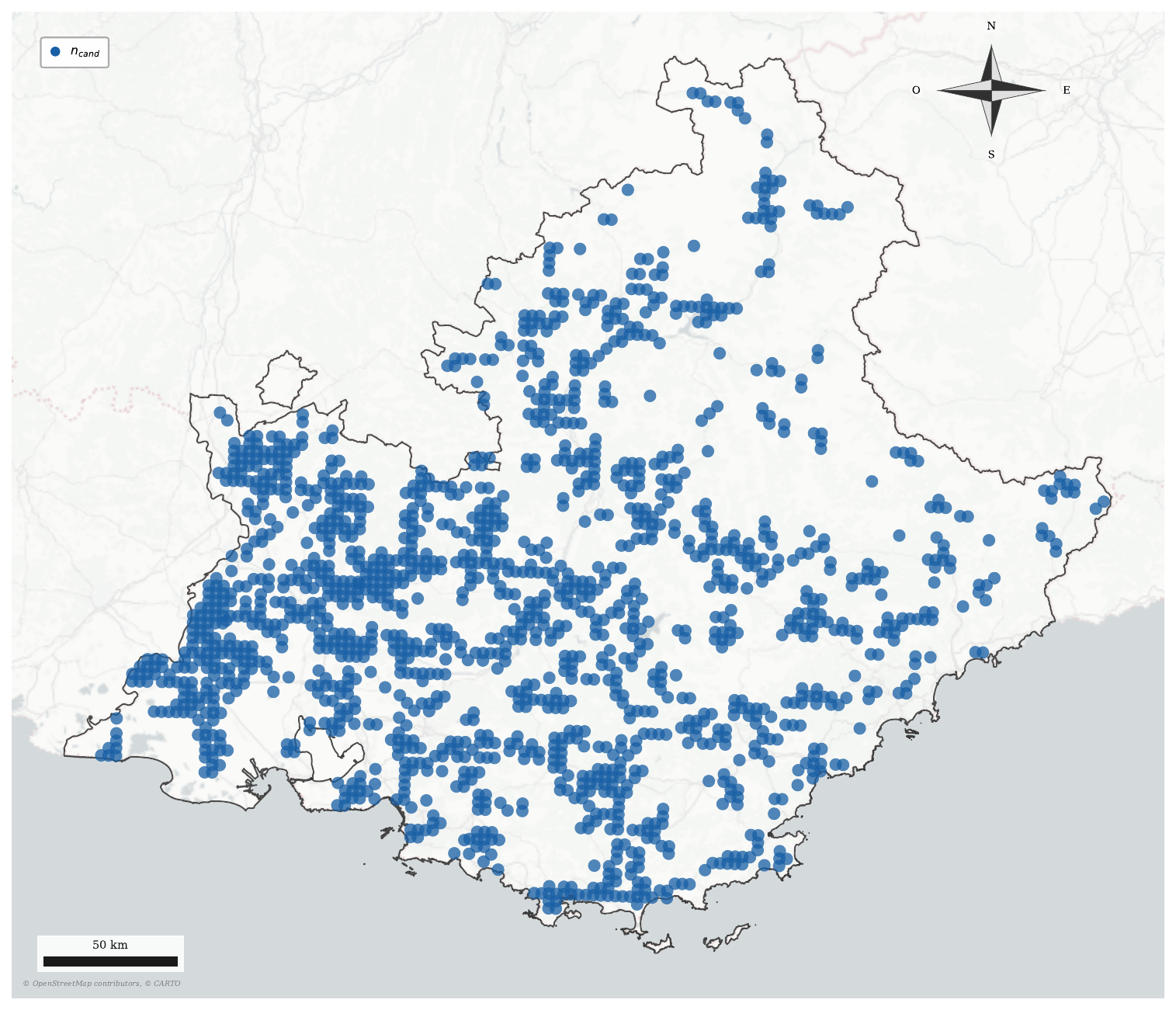}
        \caption{$\kappa = 3$}
        \label{fig:spatial_p5282_k3}
    \end{subfigure}
    
    \label{fig:study_spatial_kappa_paca5282}
\end{figure}

\begin{figure}[htbp]
    \caption{This example illustrates three different territorial divisions of the same region: $S^1$, $S^2$, and $S^3$, partitioning the area into 15, 6, and 2 spatial units, respectively. The divisions are hierarchical and satisfy the relations shown above.}
    
    \vspace{0.3cm} 
    
    {\small
    \centering
    $\begin{aligned}
    S_1^2 &= \bigcup_{k=1}^{3} S_k^1, &
    S_2^2 &= \bigcup_{k=4}^{5} S_k^1, &
    S_3^2 &= \bigcup_{k=6}^{8} S_k^1, \\
    S_4^2 &= \bigcup_{k=9}^{11} S_k^1, &
    S_5^2 &= \bigcup_{k=12}^{13} S_k^1, &
    S_6^2 &= \bigcup_{k=14}^{15} S_k^1,
    \end{aligned}$
    
    \vspace{0.2cm}
    and at the next scale,
    \vspace{0.2cm}
    
    $\begin{aligned}
    S_1^3 = \bigcup_{k=1}^{2} S_k^2,
    \qquad\qquad
    S_2^3 = \bigcup_{k=3}^{6} S_k^2.
    \end{aligned}$
    \par}

     \vspace{0.2cm}
    
    \centering
    \begin{subfigure}[t]{0.32\textwidth}
        \centering
        \includegraphics[width=1\linewidth]{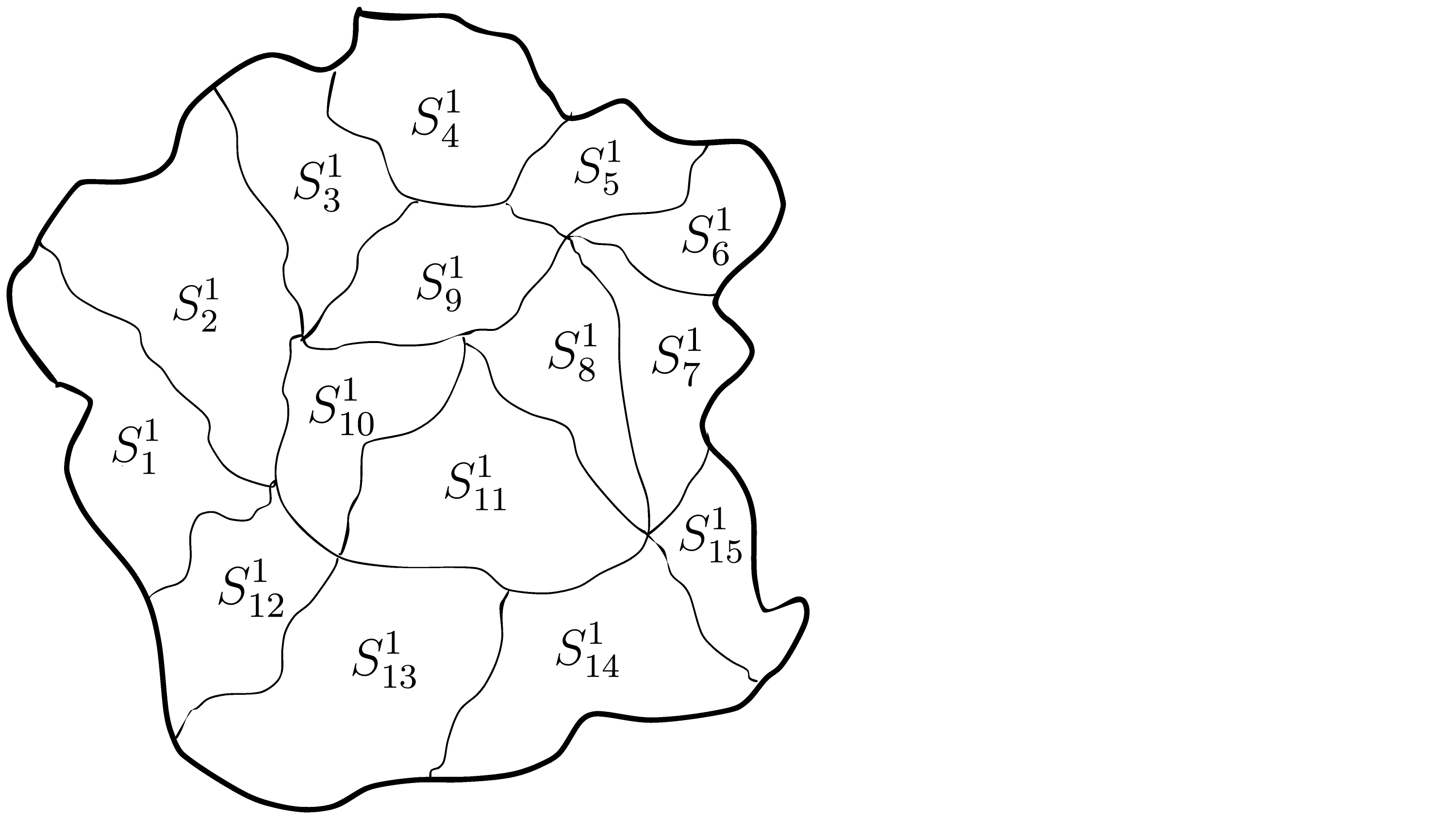}
        
        \caption{$S^1 = \{S_1^1,S_2^1, \ldots, S_{15}^1\}$}
        \label{fig:multi_s1}
    \end{subfigure}
    \hfill
    \begin{subfigure}[t]{0.32\textwidth}
        \centering
        \includegraphics[width=\linewidth]{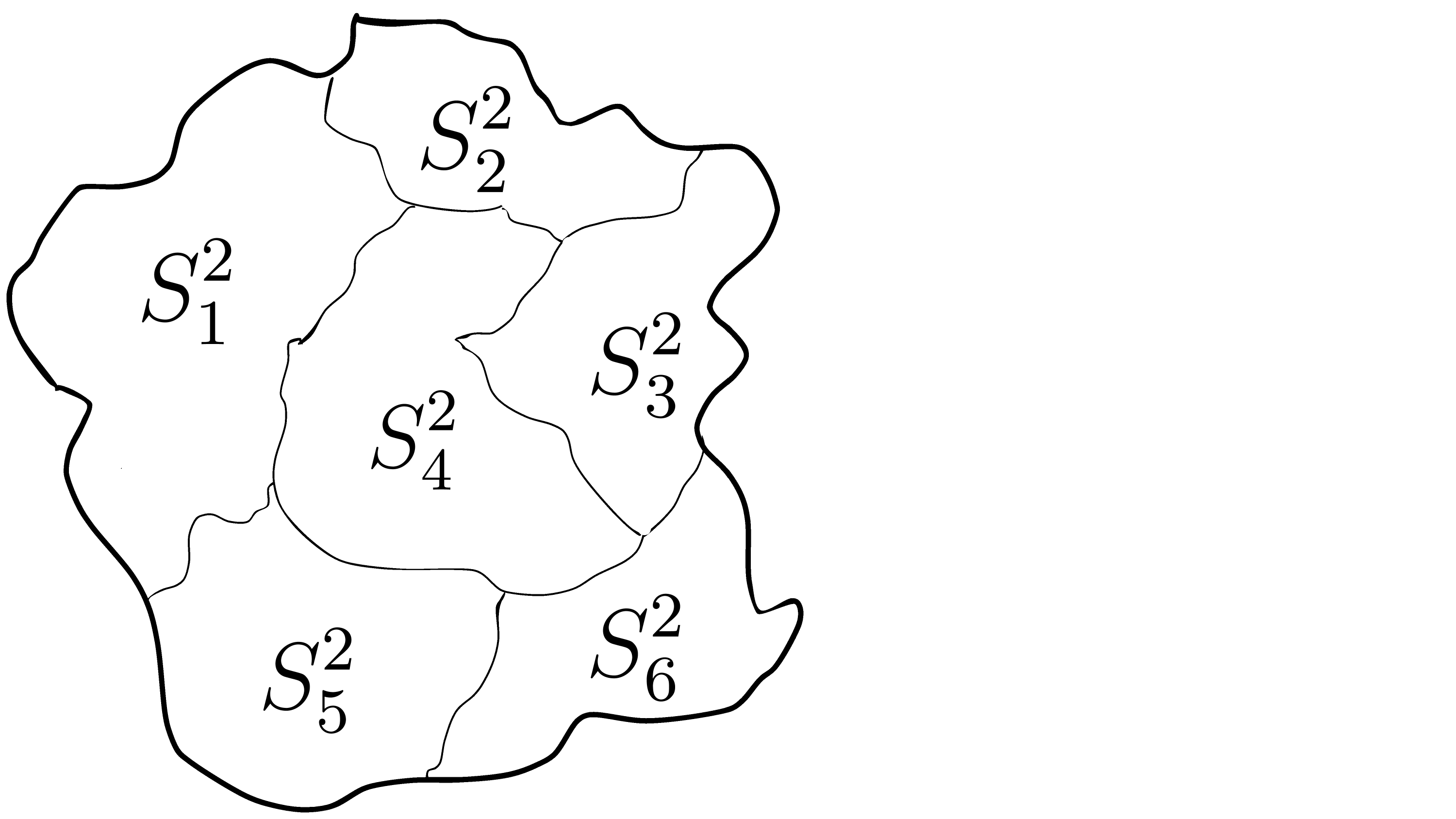}
        \caption{$S^2 = \{S_1^2,S_2^2, \ldots, S_{6}^2\}$}
        \label{fig:multi_s2}
    \end{subfigure}
    \hfill
    \begin{subfigure}[t]{0.32\textwidth}
        \centering
        \includegraphics[width=\linewidth]{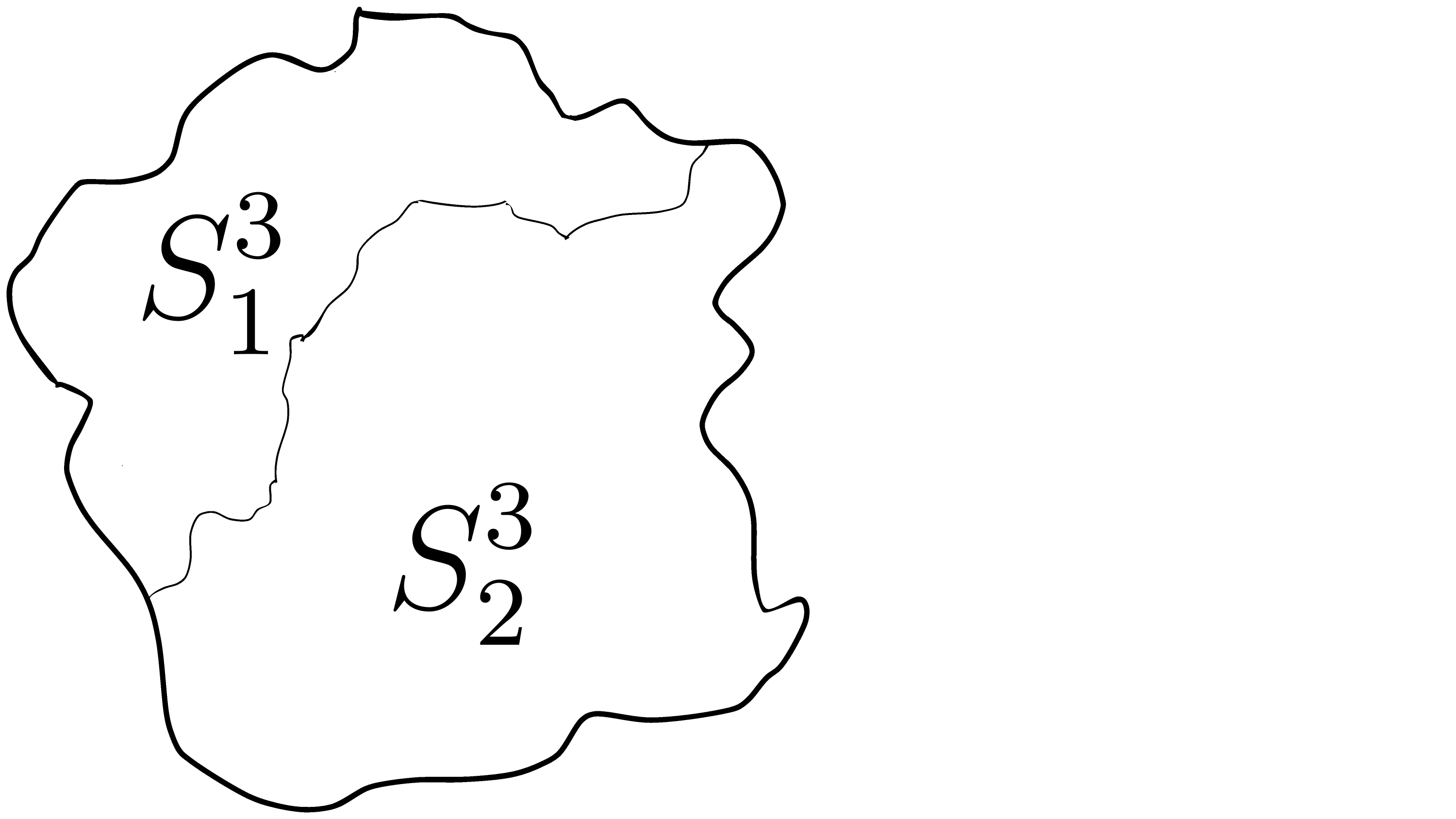}
        \caption{$S^3 = \{S_1^3,S_2^3\}$}
        \label{fig:multi_s3}
    \end{subfigure}
    
    \vspace{0.4cm} 
    
    \label{fig:multiscales_coverage_example}
\end{figure}

\begin{figure}[htbp]
    \caption{Two solution examples for different values of $p$, using the three territorial divisions of Figure~\ref{fig:multiscales_coverage_example}. Triangles denote candidate facility locations, with black triangles indicating those selected . For $p = 4$ ($|S^3| \le p \le |S^2|$), all spatial units in $S^3$ are covered, and the remaining facilities cover units in $S^2$, thereby guaranteeing the coverage of $S^3$. For $p = 8$ ($|S^2| \le p \le |S^1|$), all units in $S^2$ are covered (which also guarantees the coverage of $S^3$), the remaining facilities covering units in $S^1$.}
    \centering
    \begin{subfigure}[t]{0.48\textwidth}
        \centering
        \includegraphics[width=\linewidth]{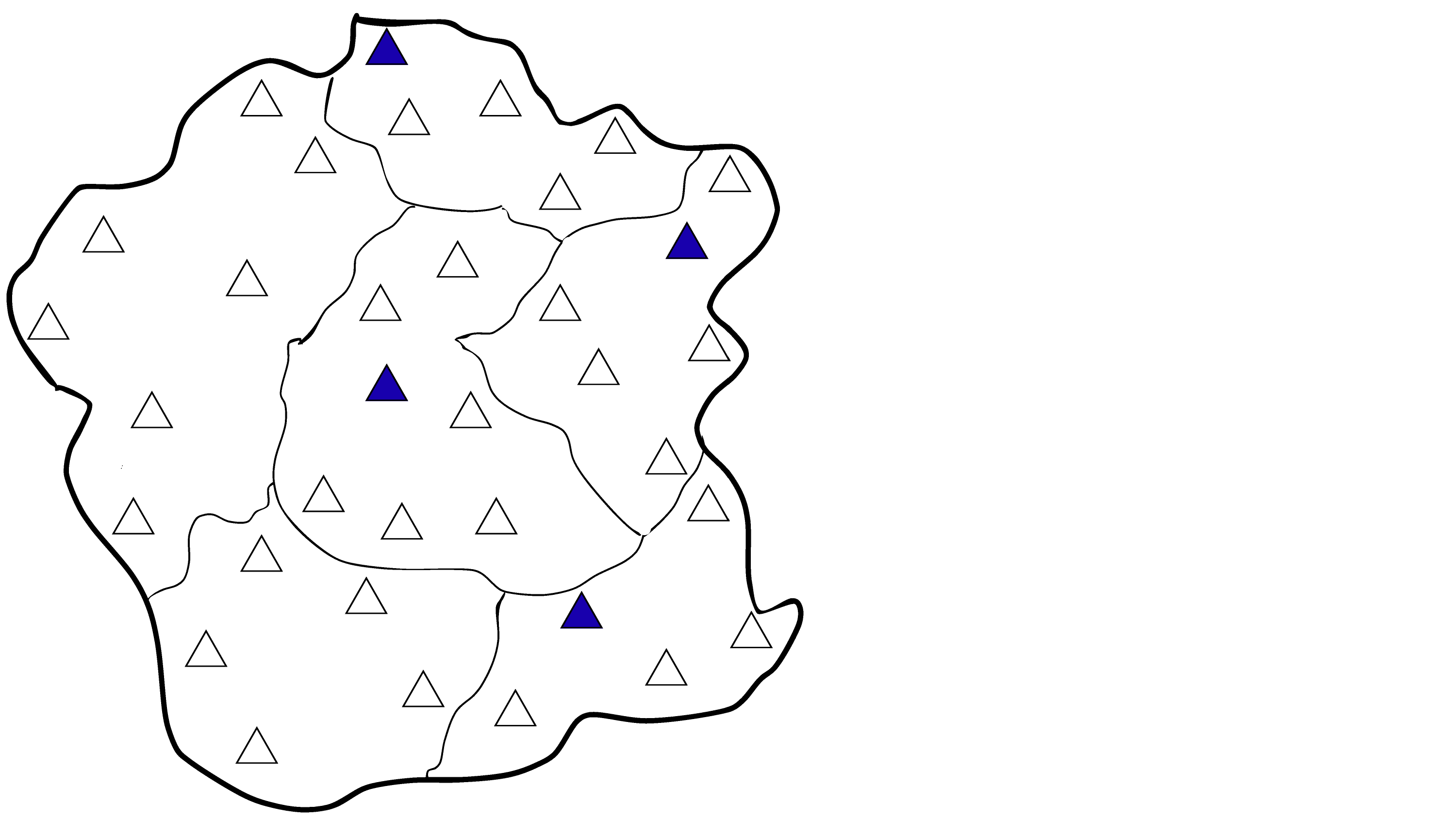}
        \caption{Solution with $p = 4$ and territorial division $S^2$}
        \label{fig:sol_p4}
    \end{subfigure}
    \hfill
    \begin{subfigure}[t]{0.48\textwidth}
        \centering
        \includegraphics[width=\linewidth]{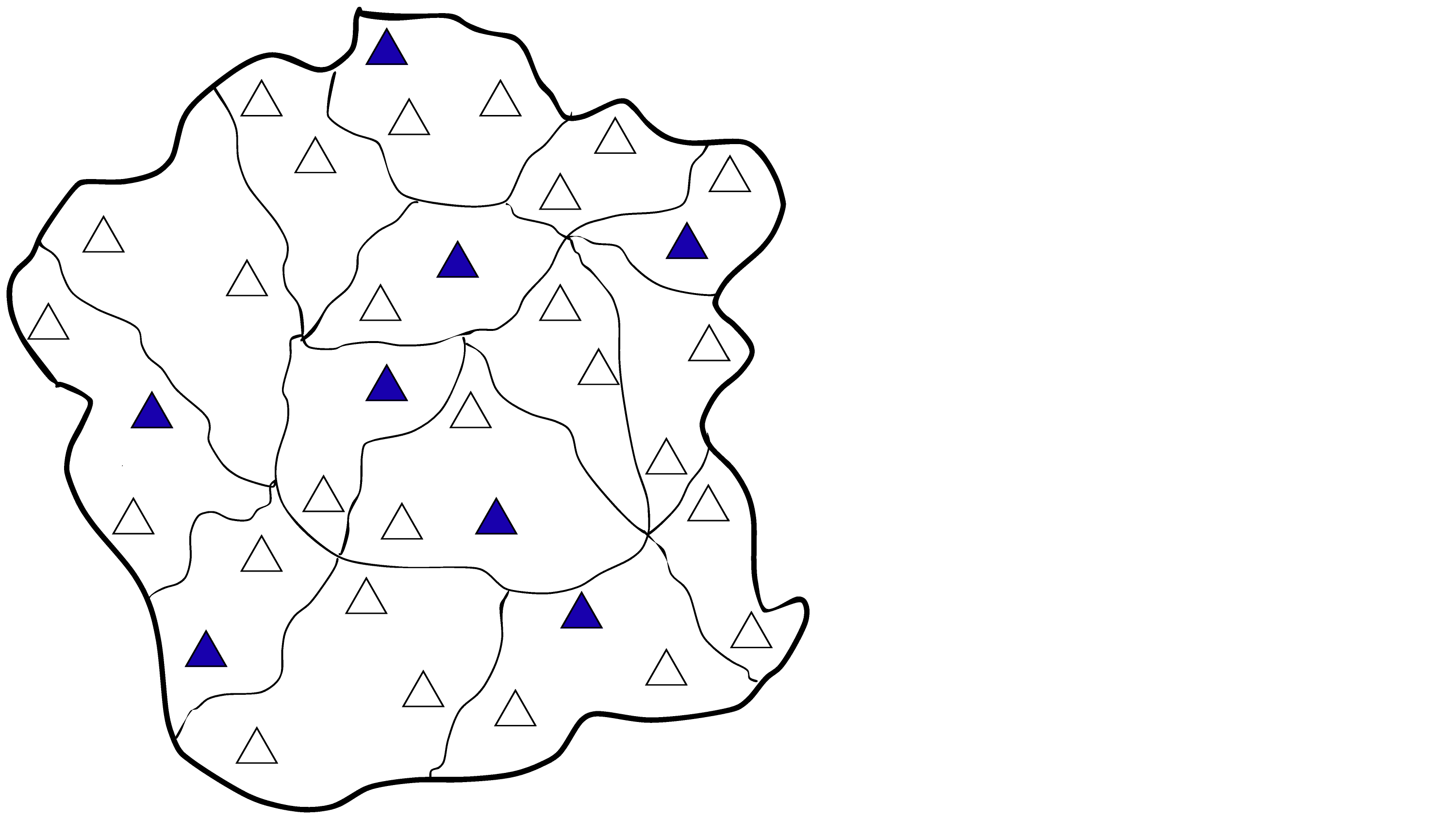}
        \caption{Solution with $p = 8$ and territorial division $S^1$}
        \label{fig:sol_p8}
    \end{subfigure}
    
    \label{fig:multiscales_coverage_solutions_example}
\end{figure}

\end{document}